\documentclass[11pt, oneside, reqno]{amsart}
\usepackage[shortlabels, inline]{enumitem}
\usepackage{mathtools, mathrsfs, xargs, amssymb}
\usepackage[dvipsnames]{xcolor}

\usepackage{stmaryrd}
\usepackage[T1]{fontenc}
\usepackage{libertine}

\usepackage{hyperref}
\hypersetup{colorlinks, citecolor=blue, linkcolor=blue, urlcolor=blue}
\usepackage[top=1.5in, bottom=1.5in, left=1.20in, right=1.20in]{geometry}

\numberwithin{equation}{section}
\theoremstyle{plain}                
\newtheorem{theorem}{Theorem}[section]
\newtheorem{lemma}[theorem]{Lemma}
\newtheorem{proposition}[theorem]{Proposition}
\newtheorem{corollary}[theorem]{Corollary}

\theoremstyle{definition}           
\newtheorem{definition}[theorem]{Definition}
\newtheorem{example}[theorem]{Example}

\newtheorem{innerassumption}{Assumption}
\newenvironment{assumption}[1]
  {\renewcommand\theinnerassumption{#1}\innerassumption}
  {\endinnerassumption}

\theoremstyle{remark}
\newtheorem{remark}[theorem]{Remark}

\newcommand{\define}[1]{{\textbf{#1}}}

\newcommand{\movealign}[1]{\hspace{-#1em}&\hspace{#1em}}

\providecommand{\alias}{}
\renewcommand{\alias}[1]{\providecommand{#1}{}\renewcommand{#1}}

\DeclarePairedDelimiter\ab{\langle}{\rangle} 
\DeclarePairedDelimiter\abs{\lvert}{\rvert}   
\DeclarePairedDelimiter\norm{\lVert}{\rVert}  
\DeclarePairedDelimiter\bkt{[}{]}             
\DeclarePairedDelimiter\brc{\{}{\}}           
\DeclarePairedDelimiter\prn{(}{)}             

\providecommand\giv{}

\alias\given\giv 

\DeclarePairedDelimiterXPP\pp[1]{\mathbb{P}}[]{}{
   \renewcommand\giv{\nonscript\:\delimsize\vert\nonscript\:\mathopen{}}
   #1}
\DeclarePairedDelimiterXPP\ppup[2]{\mathbb{P}^{#1}}[]{}{
   \renewcommand\giv{\nonscript\:\delimsize\vert\nonscript\:\mathopen{}}
   #2}
\DeclarePairedDelimiterXPP\ppdown[2]{\mathbb{P}_{#1}}[]{}{
   \renewcommand\giv{\nonscript\:\delimsize\vert\nonscript\:\mathopen{}}
   #2}
   \DeclarePairedDelimiterXPP\ppupdown[3]{\mathbb{P}^{#1}_{#2}}[]{}{
   \renewcommand\giv{\nonscript\:\delimsize\vert\nonscript\:\mathopen{}}
   #3}
   
\DeclarePairedDelimiterXPP\ee[1]{\mathbb{E}}[]{}{
   \renewcommand\giv{\nonscript\:\delimsize\vert\nonscript\:\mathopen{}}
   #1}

\DeclarePairedDelimiterXPP\var[1]{\mathrm{Var}}[]{}{
   \renewcommand\giv{\nonscript\:\delimsize\vert\nonscript\:\mathopen{}}
   #1}

 \DeclarePairedDelimiterXPP\cov[1]{\mathrm{Cov}}[]{}{
   \renewcommand\giv{\nonscript\:\delimsize\vert\nonscript\:\mathopen{}}
   #1}

\DeclarePairedDelimiterXPP\eeup[2]{\mathbb{E}^{#1}}[]{}{
   \renewcommand\giv{\nonscript\:\delimsize\vert\nonscript\:\mathopen{}}
   #2}

\DeclarePairedDelimiterXPP\eedown[2]{\mathbb{E}_{#1}}[]{}{
   \renewcommand\giv{\nonscript\:\delimsize\vert\nonscript\:\mathopen{}}
   #2}

   \DeclarePairedDelimiterXPP\eeupdown[3]{\mathbb{E}^{#1}_{#2}}[]{}{
   \renewcommand\giv{\nonscript\:\delimsize\vert\nonscript\:\mathopen{}}
   #3}
\DeclarePairedDelimiterXPP\eeud[3]{\mathbb{E}^{#1}_{#2}}[]{}{
   \renewcommand\giv{\nonscript\:\delimsize\vert\nonscript\:\mathopen{}}
   #3}

\let\preexp\exp
\let\exp\relax

\DeclarePairedDelimiterXPP\exp[1]{\preexp}(){}{#1}

\newcommand{\Exp}[1]{e^{#1}}

\newcommandx{\prf}[2][2=t\in\bkt{0,T}]{ \set{#1}_{#2}}

\newcommandx{\seq}[2][2=n\in\N]{\brc*{#1}_{#2}}
\newcommandx{\sseq}[2][2=n\in\N]{\{#1\}_{#2}}

\newcommand{\seqz}[1]{\seq{#1}[n\in\N_0]}

\newcommandx{\seqk}[1]{\seq{#1}[k\in\N]}
\newcommandx{\seqkz}[1]{\seq{#1}[k\in\N_0]}

\newcommandx{\seqj}[1]{\seq{#1}[j\in\N]}
\newcommandx{\seqjz}[1]{\seq{#1}[j\in\N_0]}

\newcommandx{\seqi}[1]{\seq{#1}[i\in\N]}
\newcommandx{\seqiz}[1]{\seq{#1}[i\in\N_0]}

\newcommandx{\seqm}[1]{\seq{#1}[m\in\N]}
\newcommandx{\seqmz}[1]{\seq{#1}[m\in\N_0]}

\newcommandx{\sseqk}[1]{\sseq{#1}[k\in\N]}
\newcommandx{\sseqkz}[1]{\sseq{#1}[k\in\N_0]}

\newcommandx{\fml}[2]{\set{#1}_{#2}}
\newcommandx{\tol}[1]{\stackrel{#1}{\longrightarrow}}
\newcommandx{\eqd}[1][1=d]{\stackrel{(#1)}{=}}
\newcommandx{\neqd}[1][1=d]{\stackrel{(#1)}{\neq}}

\makeatletter
\let\oldabs\abs \def\abs{\@ifstar{\oldabs}{\oldabs*}}
\let\oldab\ab \def\ab{\@ifstar{\oldab}{\oldab*}}
\let\oldnorm\norm \def\norm{\@ifstar{\oldnorm}{\oldnorm*}}
\let\oldbkt\bkt \def\bkt{\@ifstar{\oldbkt}{\oldbkt*}}
\let\oldbrc\brc \def\brc{\@ifstar{\oldbrc}{\oldbrc*}}
\let\oldprn\prn \def\prn{\@ifstar{\oldprn}{\oldprn*}}
\let\oldpp\pp \def\pp{\@ifstar{\oldpp}{\oldpp*}}
\let\oldppup\ppup \def\ppup{\@ifstar{\oldppup}{\oldppup*}}
\let\oldppdown\ppdown \def\ppdown{\@ifstar{\oldppdown}{\oldppdown*}}
\let\oldee\ee \def\ee{\@ifstar{\oldee}{\oldee*}}
\let\oldvar\var \def\var{\@ifstar{\oldvar}{\oldvar*}}
\let\oldeeup\eeup \def\eeup{\@ifstar{\oldeeup}{\oldeeup*}}
\let\oldeedown\eedown \def\eedown{\@ifstar{\oldeedown}{\oldeedown*}}
\let\oldeeupdown\eeupdown \def\eeupdown{\@ifstar{\oldeeupdown}{\oldeeupdown*}}
\let\oldeeud\eeud \def\eeud{\@ifstar{\oldeeud}{\oldeeud*}}

\let\oldexp\exp \def\exp{\@ifstar{\oldexp}{\oldexp*}}
\makeatother

\newcommand{\Bee}[1]{ \ee*[\Big]{#1} }
\newcommand{\bee}[1]{ \ee*[\big]{#1} }

\newcommand{\Bprn}[1]{ \prn*[\Big]{#1} }
\newcommand{\bprn}[1]{ \prn*[\big]{#1} }

\newcommand{\Bbrc}[1]{ \brc*[\Big]{#1} }

\newcommand{\bexp}[1]{\exp*[\big]{#1}}

\newcommand{\babs}[1]{\abs*[\big]{#1}}

\newcommand{\set}[1]{ \brc{#1} }
\newcommand{\sets}[2]{ \brc{#1\,:\,#2} }

\newcommand{\Bsets}[2]{ \Bbrc{#1\,:\,#2} }

\alias{\R}{{\mathbb R}}
\alias{\C}{{\mathbb C}}
\alias{\Z}{{\mathbb Z}}
\alias{\N}{{\mathbb N}}
\alias{\Nz}{{\mathbb N}_0}

\DeclareMathOperator\Cov{Cov}

\newcommand{\oo}[1]{\frac{1}{#1}}
\newcommand{\ot}{\oo{2}} 
\newcommand{\too}[1]{\tfrac{1}{#1}}
\newcommand{\tot}{\too{2}} 

\newcommand{\ind}[1]{ 1_{{#1}}} 
\newcommand{\inds}[1]{ 1_{\set{#1}}} 

\newcommand{\downto}{\searrow}
\newcommand{\upto}{\nearrow}

\newcommand{\tsum}{\textstyle\sum}
\newcommand{\tprod}{\textstyle\prod}

\newcommand{\al}{\alpha}
\newcommand{\be}{\beta}
\newcommand{\dl}{\delta}
\newcommand{\eps}{\varepsilon}
\newcommand{\ld}{\lambda}

\newcommand{\gm}{\gamma}
\newcommand{\vp}{\varphi}

\newcommand{\sB}{\mathcal{B}}

\newcommand{\sL}{\mathcal{L}} \newcommand{\bL}{\mathbb{L}}
\newcommand{\sM}{\mathcal{M}}

\newcommand{\sS}{\mathcal{S}}

\newcommand{\sY}{\mathcal{Y}}

 \newcommand{\hF}{\hat{F}}
 
\newcommand{\tH}{\tilde{H}}

\newcommand{\tP}{\tilde{P}}

\newcommand{\lone}{\bL^1}

\newcommandx{\pds}[2][1= ]{\frac{\partial #1}{\partial #2}}
\newcommandx{\tpds}[2][1= ]{\tfrac{\partial #1}{\partial #2}}
\newcommandx{\dd}[2][1= ]{\frac{d #1}{ d#2}}
\newcommandx{\tdd}[2][1= ]{\tfrac{d #1}{ d#2}}
\newcommandx{\pdt}[2][1= ]{\frac{\partial^2 #1}{\partial #2}}
\newcommandx{\tpdt}[2][1= ]{\tfrac{\partial^2 #1}{\partial #2}}

\newcommandx{\pdm}[3][2= ]{\frac{\partial^{#1} #2}{\partial #3}}
\newcommandx{\tpdm}[4][2= ]{\tfrac{\partial^{#1} #2}{\partial #3}}

\newcommandx{\up}[2]{{#1}^{#2}}
\newcommandx{\upp}[2]{{#1}^{(#2)}}

\providecommand{\upm}[1]{\upp{#1}{m}}

\newcommand{\efor}{\text{ for }}
\newcommand{\eforall}{\text{ for all }}

\newcommand{\eand}{\text{ and }}

\newcommand{\ewhere}{\text{ where }}
\newcommand{\ewith}{\text{ with }}

\newcommand{\dlink}[2]{\href{http://dlmf.nist.gov/#1}{#2}}
\newcommand{\cdg}[2]{\cite[\dlink{#1}{#2}]{DLMF}}

\renewcommand{\define}[1]{\emph{#1}}

\newcommand{\mgf}{M}

\newcommand{\apc}{\approx}
\newcommand{\drtr}{\tfrac{1-\dl}{2} }
\newcommand{\hmu}{\hat{\mu}} 
\newcommand{\hnu}{\hat{\nu}}
\newcommand{\tmod}[1]{ \,(\text{mod } #1)}
\newcommand{\slop}{\sL^p_{\mathrm{loc}}}

\newcommand{\arr}[1]{\{#1\}_{m,n\in\N}}
\newcommand{\ceil}[1]{\lceil #1 \rceil}
\newcommand{\ct}{\ceil{\tau/\eps_n}}
\newcommand{\cs}{\ceil{\sigma/\eps_n}}
\newcommand{\cz}{\ceil{\zeta/\eps_n}}
\newcommand{\co}{\ceil{1/\eps_n}}
\newcommand{\tct}{\ceil{\tfrac{\tau}{\eps_n}}}
\newcommand{\tcs}{\ceil{\tfrac{\sigma}{\eps_n}}}

\newcommand{\cat}{m_n(\tau)}
\newcommand{\cas}{m_n(\sg)}

\newcommand{\sg}{\sigma}
\newcommand{\te}{\tilde{e}}
\newcommand{\txi}{\tilde{\xi}}

\newcommand{\hrho}{\hat{\rho}}

\newcommand{\hgF}{{}_2 F_1}
\newcommand{\hpi}{\hat{\pi}}

\newcommand{\tg}{\tilde{g}}
\renewcommand{\upm}[1]{{#1}^m}
\newcommand{\utH}[1]{\up{\tH}{#1}}
\newcommand{\utP}[1]{\up{\tP}{#1}}
\newcommand{\upi}[1]{\up{\pi}{#1}}

\newcommand{\seqmn}[1]{\{#1\}_{m,n\in\N}}

\newcommand{\zT}{[0,T]}
\newcommand{\zt}{[0,t]}
\newcommand{\pos}{[0,\infty)}

\newcommand{\smlf}{\sM_{\mathrm{lf}}}
\newcommand{\smp}{\sM_{\mathrm{p}}}
\newcommand{\sms}{\sM_{\mathrm{s}}}

\newcommand{\smf}{\sM_{\mathrm{f}}}

\newcommand{\lap}{\widehat{\ }}
\newcommand{\hh}{\hat{h}}
\newcommand{\tih}{\tilde{h}}
\newcommand{\slo}{\slone}
\newcommand{\slone}{\sL^1}

\newcommand{\ssi}{\sS^{\infty}}
\newcommand{\ssiT}{\ssi\zT}
\newcommand{\ssit}{\ssi\zt}
\newcommand{\ssil}{\ssi_{\mathrm{loc}}}
\newcommand{\bv}{\mathrm{BV}}

\newcommand{\bvl}{\bv_{\mathrm{loc}}}

\newcommand{\woT}{W^1_{\zT}}
\newcommand{\n}[1]{\abs{#1}}
\newcommand{\ns}[1]{\abs*{#1}}

\newcommand{\pimk}{\pi^{(m,m+k]}}
\newcommand{\pimkp}{\pi^{(m,m+k+1]}}

\newcommand{\ahnn}{\abs*{\tH_n}}

\begin{document}
\title[Convergence of Nonhomogeneous Hawkes Processes]{Convergence of 
  nonhomogeneous Hawkes processes and Feller random measures} 
\author{Tristan Pace}
\address{Tristan Pace, Department of Mathematics, The University of Texas at
  Austin}
\email{tpace4288@utexas.edu}

\author{Gordan {\v Z}itkovi{\' c}} \address{Gordan {\v Z}itkovi{\' c},
Department of Mathematics, The University of Texas at Austin}
\email{gordanz@math.utexas.edu}

  \begin{abstract} 
    We consider a sequence of Hawkes processes whose excitation measures
    may depend on the generation, and study its scaling limits in the
    near-unstable limiting regime. The limiting random measures,
    characterized via a nonlinear convolutional equation, form a family
    parameterized by a pair consisting of a locally finite measure and a
    geometrically infinitely divisible probability distribution on the
    positive real line. These measures can be interpreted as
    generalizations of the Feller diffusion and fractional Feller (CIR)
    processes, but also allow for a "driving noise" associated with general
    L\' evy-type operators of order at most $1$, including fractional
    derivatives of any order $\alpha>0$ (formally corresponding to possibly
    negative Hurst parameters). 
\end{abstract}

\keywords{Hawkes processes, functional central
  limit theorem, nearly unstable, convolutional Riccati
  equation, fractional CIR processes, geometric
infinite divisibility, Feller random measures} 

\thanks{
During the preparation of this work both authors were 
supported by the National Science Foundation 
under Grant DMS-2307729. Any opinions, findings and
conclusions or recommendations expressed in this material are
those of the authors and do not necessarily reflect the views of the
National Science Foundation (NSF).
}

\subjclass[2020]{60F17, 60G55, 60G57, 45D05}

\maketitle

\section{Introduction}

Hawkes point processes were introduced in
\cite{Haw71,Haw71a} as models for self-exciting
stochastic phenomena. Their fundamental property is that new points
are generated at a rate that depends on the number and locations of
existing points via a function known as the excitation intensity.
Initially used as models for seismic events, Hawkes processes have
since found numerous applications in various disciplines ranging from
epidemiology and criminology, to genetics, neuroscience, 
economics and finance (see the survey \cite{LauLeePolTai24} and its
references).

\subsection{Limiting theory of Hawkes processes - an overview of the
literature}
The investigation into the limiting theory of Hawkes processes began
almost immediately after their introduction.  A central limit theorem
(as $t\to\infty$) for Hawkes processes whose kernels admit a finite
first moment was already established in \cite{HawOak74} (this paper
also introduced the cluster representation we use in the current paper). We
start with a brief survey of existing pertinent results split
into two classes, based on the scaling regime.

In the first class, the total mass $a=\int_0^{\infty} \phi(t)\, dt$ of
the kernel $\phi$ is kept constant, while time, space and other
parameters are scaled. One of the earliest results here was provided by
\cite{BacDelHof13a}, where a functional central limit theorem (FCLT)
with convergence to a scaled Brownian motion was established
under a finiteness assumption on the $1/2$-th moment of the kernel.
Later, \cite{GaoZhu18a} introduced a framework where the background
(immigrant) intensity is taken to infinity, but only space is scaled
to compensate. Under the assumption that the kernel is exponential, it
is shown there that the limiting process is no longer Brownian but
only Gaussian with a non-Markovian covariance function. More recently,
a FCLT for marked Hawkes processes and associated shot noises was
established in \cite{HorXu21}. We also mention a recent preprint
\cite{HorXu24} by the same authors where FCLTs or parallel negative
results are established for Hawkes processes in several stability
regions defined by the values of the total mass and the first moment
of the kernel $\phi$.

The second class of results features the "nearly unstable" scaling regime,
introduced in \cite{JaiRos15}, which is also utilized in the present work.
In this regime, both time and space are scaled in a non-Brownian manner,
while the total mass of the kernel is sent to $1$, the stability threshold
constant. Assuming that the kernel has a finite first moment, these authors
establish a functional scaling limit theorem for the integrated intensity
process with the Feller (CIR) diffusion, a non-Gaussian process, as the
limit. In the follow-up paper \cite{JaiRos16}, the requirement for a finite
first moment is relaxed to the finiteness of some moment above $1/2$, and
the limiting Feller diffusion is replaced by a fractional Feller (CIR)
process, which is neither Markovian nor a semimartingale. The mode of
convergence obtained in \cite{JaiRos16} was strengthened in
\cite{HorXuZha23} under the same assumptions on the kernel. These authors
show that the intensity processes themselves converge in the Skorokhod
topology, and not only their integrals, as in \cite{JaiRos16}.

\subsection{Our contributions} The goal of this paper is to add to the
existing literature by extending the aforementioned results in several
directions. Firstly, we consider  nonhomogeneous Hawkes processes, i.e.,
generalizations of Hawkes processes where the kernel is allowed to vary
from generation to generation (see \cite{FieLeiMol15} or \cite{MehZhu25} 
for a related model).
This not only provides additional modeling flexibility, but also unlocks a
wider range of possible limiting objects. Additionally, we allow the
kernels themselves to serve as scaling parameters in that they may depend
on the scaling parameter $n$. Compared to the existing results, these two
extensions can be interpreted as a transition from scaled sums of
i.i.d.~sequences to sums of (triangular) arrays of independent random
variables in classical probability theory. Continuing this analogy, the
class of our limiting objects now includes not only the analogues of stable
distributions (fractional Feller (CIR) processes), but also the analogues
of infinitely divisible distributions (termed Feller random measures in
this paper). 

Another direction in which we broaden all existing results is that we do
not impose any conditions on the integrability of the kernels; we do not
even require them to be functions in $\slone\pos$ but permit them to be
general finite measures on $\pos$. This allows us, in particular, to expand
the analysis of \cite{JaiRos15, JaiRos16} down to and below the critical
$1/2$-moment threshold imposed in the existing literature. In this regime,
the limiting objects are no longer necessarily (integrals of) stochastic
processes; they can now be located only in the space of locally finite random
measures. Hence, we cannot talk about convergence in Skorokhod's $J_1$, or
any related topology, but instead work with the vague topology this space
is naturally endowed with. On the other hand, we show that the limiting
random measures are almost surely atom-free; vague convergence in that case
implies weak convergence of cumulative distribution processes in the
topology of locally uniform (and therefore locally $J_1$) convergence.

\smallskip

The level of generality of our framework and the appearance of limiting
objects beyond the class of stochastic processes render the tools of
stochastic analysis and martingale theory, standard in the Hawkes-process
literature, less effective. This is further exacerbated by the
nonhomogeneity of our model; the generation-dependence of the excitation
kernel makes it challenging to express the conditional intensity process in
a convenient form without sacrificing finite dimensionality. Consequently,
we are led to the cluster representation of the Hawkes process (similar to
\cite{ElERos19}) and the related cascade of relationships among the Laplace
functionals associated with a sequence of auxiliary point processes. Here,
the process indexed by $m$ represents the progeny of an individual of
generation $m-1$. The crux of the argument then rests on obtaining tight
joint estimates for processes of different indices.
This results in a
convergence theorem that provides scaling constants, gives sufficient
conditions under which the scaling limit exists, and characterizes its
Laplace functional in terms of the unique solution to a nonlinear
convolutional Riccati equation. 
The central condition on the array of excitation kernels takes the flavor
of classical convergence of triangular arrays and is given in two equivalent
formulations. Moreover, three alternative natural and easily verifiable 
conditions for the main theorem to hold are provided and discussed.

We also show that all possible limits of sequences of scaled 
Hawkes processes in our setting are 
completely described by two measure-valued parameters: a locally finite
measure $\mu$ and a 
probability measure $\rho$ on $\pos$. While $\mu$ can be any locally finite
measure on $\pos$, the class of possible $\rho$ turns out to 
coincide with the 
class of \emph{geometrically infinitely divisible distributions}. Such
distributions admit a theory somewhat parallel to that of 
infinitely divisible distributions (see, e.g., \cite{KorKru93}) and 
the appearance of a particular $\rho$ in our limit can be directly related
to the characteristics of the array of the excitation measures, especially
in the cases of \emph{geometrically stable} distributions, i.e.,
distributions in the Mittag-Leffler family. 

The name \emph{Feller random measures} is chosen for our limiting objects
because, as shown in \cite{JaiRos15} and \cite{JaiRos16a},  their densities
are given by a Feller diffusion when $\rho$ is the exponential distribution
and a rough (fractional) Feller process when $\rho$ is the Mittag-Leffler
distribution with index above $1/2$. The third section of the paper is devoted 
to studying the distributional properties of Feller random measures, such as
moments, infinite divisibility and (non) existence of densities in 
special cases.

\subsection{Connections with fractional Brownian motion and rough
volatility models} One of the motivations for this work comes from the role
Hawkes processes have in financial modeling. Their self-exciting nature is
particularly well-suited for capturing the dependence of market buy and
sell orders on past orders (see, e.g.,  \cite{BacMasMuz15} for an
overview). A phenomenon well-explained by such models is the observed
"roughness" (see \cite{ComRen98} and \cite{GatJaiRos18}) of market
volatility. Indeed, the fractional Feller (CIR) process that appears in the
results of \cite{JaiRos16} - and corresponds to the squared volatility -
can be informally thought of as a continuous stochastic process "driven" by
a fractional Brownian motion (fBM). The value of the Hurst parameter $H \in
(0,1)$ of this fBM, used to describe the degree of "roughness" of the
volatility process, has been the subject of several empirical studies.
Early estimates gave $H\in(1/2,1)$ (\cite{ComRen98}) whereas two decades
later the consensus shifted towards $H\in(0,1/2)$ (see \cite{GatJaiRos18},
\cite{BenLunPak21} and \cite{FukTakWes22}). Many of the later estimates
put $H$ very close to $0$, suggesting that $H=0$ might be the "true" value
(see \cite{ForFukGerSmi22}, \cite{BayHarPig21}).

Even though there is no universally accepted way to define a fractional
Brownian motion with $H=0$ either as a stochastic process or as a random
measure/field, several authors have proposed models that could play such a
role in one sense or another. These include the multifractal random walks
(see \cite{BacDelMuz01}) and various Gaussian random fields with a
logarithmic kernel (see \cite{FyoKhoSim16}, \cite{NeuRos18}, and
\cite{HagNeu22} for a sample of different approaches). Our framework allows
us not only to define generalized fractional Feller (CIR) processes
corresponding to values of the Hurst parameter $H$ in the interval
$(-1/2,1/2]$, but also to encompass a much wider range of driving noises
beyond the one-dimensional fractional family. Moreover, we only require a
single passage to the limit, and do not define a limiting process for $H>0$
first, and then pass it to a (second) limit $H\to 0$, as is often done in
the literature mentioned above. While a full analysis is left for future
research, and it is difficult to give a formal definition of the notion of
a driving noise for random measures, we do note that the form of the
covariance kernel that we obtain in subsection \ref{sss:cov} below suggests
the log-correlated class (see the survey \cite{DupRhoSheVar17}) in the case
$H=0$. 

\subsection{Organization of the paper} Following this introduction, Section 2
introduces the necessary background and notation, and states the main theorem
together with several weaker sufficient conditions for its validity. Section 3
is devoted to the study of various properties of the limiting Feller random
measures. All proofs are deferred to the appendix, which is organized into four
subsections. 
Subsection
\ref{sec:proof-asm} provides proofs related to alternative conditions for the main theorem, Subsection \ref{sec:Riccati} develops results on convolutional Riccati
equations, and Subsection \ref{sec:aux-Hawkes} contains additional 
estimates related to Hawkes processes. Finally, Section \ref{sec:proof-main} combines these
ingredients to establish the main theorem through several auxiliary results.

\subsection{Notation and conventions}
\label{ssc:notation}

Let $I$ be either  $\pos$, $(0,\infty)$, or $[0,T]$ for some $T \geq 0$.
$\ssi(I)$ denotes  the family of all Borel measurable functions on $I$ such
that $\n{f}_{\ssi(I)}:= \sup_{t\in I} \abs{f(t)} < \infty$, while
$\ssil\pos$ consists of all functions $f:\pos \to \R$ such that $f|_{[0,T]}
\in \ssi[0,T]$ for each  $T\geq 0$. 

$\sB(I)$ denotes the set of Borel subsets of $I$, and $\sM(I)$ denotes the
set of all positive Borel measures on $I$. The measurable structure on
$\sM(I)$ is generated by the evaluation maps $\mu \mapsto \mu(A)$, with $A
\in \sB(I)$. The sets of finite measures, measures that are finite on
bounded sets, and probability measures on $I$ are denoted by $\smf(I)$,
$\smlf(I)$, and $\smp(I)$, respectively. The total mass of
$\mu\in\sM_{\cdot}(I)$ is denoted by $\ns{\mu}$.

For $\mu \in \smlf\pos$, $\hmu$ denotes its Laplace transform $\hmu(\ld) =
\int_{\pos} e^{-\ld t}\, \mu(dt)$. The convolution of $f\in \ssil\pos$  and
$\mu\in \smlf\pos$, defined by $ \int_{[0,\cdot]} f(\cdot-s)\, \mu(ds)$, is
denoted by $f*\mu$.

We recall that the \define{vague topology} on $\smlf(I)$, with convergence
denoted by $\tol{v}$,  is the coarsest topology on $\smlf(I)$ such that the
map $\mu \mapsto \int f\, d\mu$ is continuous  for each continuous function
$f : I \to \R$ with compact support. The \define{weak topology} on the
space $\smf\pos$, with convergence denoted by $\tol{w}$,  is defined
similarly, but with the requirement of a compact support relaxed to
boundedness of the test functions $f$. 

Random elements taking values in $\sM(I)$ are called \define{random
measures}, and random measures with values in $\smlf(I)$ are said to be
\define{locally finite}.
For a sequence of locally finite random measures $\seq{\xi_n}$, we say that
$\xi_n$ \define{converges to $\xi$ in distribution}, and write $\xi_n
\Rightarrow \xi$,  if $\xi_n$ converges weakly to $\xi$ when interpreted as
a sequence of random elements in $\smlf(I)$ equipped with the vague
topology. We refer the reader to \cite[Section 4.1, p.~111]{Kal17} for a
textbook treatment and proofs of properties of the vague and weak
topologies used in the paper.

$\ceil{\cdot}$ denotes the ceiling function, i.e., $\ceil{x}$ is the
smallest integer greater than or equal to $x$.

\section{A convergence theorem for nonhomogeneous Hawkes processes}
\label{sec:main} Before we present our main result, we introduce the
notation and adopt one of several similar not entirely equivalent
frameworks for Hawkes processes found in the literature. For background
information  on random measures and point processes, we refer the reader to
standard texts such as \cite{DalVer03,DalVer08} or \cite{Kal17}.

\subsection{Hawkes processes}
A locally finite random measure $N$ on $\pos$ is called a
\define{point process} if
$N(A) \in \N_0$ for every bounded measurable set $A \subseteq \pos$.
Each point process $N$ admits a
sequence $\seqk{T_k}$ of $[0,\infty]$-valued random variables,
called the \define{points of $N$}, such that $T_0\leq T_1\leq \dots$ and
$T_k\to\infty$, a.s., and
\begin{align*}
  N =  \tsum_{k} \dl_{T_k},
\end{align*}
where $\delta_t$ denotes the Dirac measure at $t$, and
the sum is taken only over $k$ with $T_k<\infty$;
equivalently,
$\dl_{+\infty}$ is identified with the zero measure on $\pos$.
Since $\int f(t)\, N(dt) = \tsum_k f(T_k)$, a.s., whenever
both sides are well defined,
we frequently use the shorthand
\begin{align*}
  \tsum_{T \in N} f(T) := \int f(t)\, N(dt).
\end{align*}

Recall that for $\mu \in \smlf\pos$, the  \define{Poisson process with
intensity measure $\mu$} is the unique point process $P$ such that \ 1)
$P(A)$ is Poisson distributed with mean $\mu(A)$ for each bounded 
$A\in \sB(\pos)$, and \  2) $P(A_1)$, \dots, $P(A_n)$ are independent
random variables whenever $A_1,\dots, A_n$ are measurable, bounded and
disjoint.

\smallskip

The definition of a standard Hawkes process involves two components: the
background arrival rate and the excitation function. It will be convenient
for our later analysis to separate the two and first construct a class of
processes without any background intensity, but started, instead, from a
single point (progenitor) at time $t=0$. In our nonhomogeneous setting, its
distribution is determined by two parameters: a constant $a\in (0,1)$ and a
sequence $\pi=\seqm{\upm{\pi}}$ of \emph{excitation profiles}, i.e.,
probability measures on $(0,\infty)$. Together, they form the sequence
$\seqm{ a \upm{\pi}}$ of \emph{excitation measures}. To connect this
formulation with the standard notation, observe that when $\pi^m$ is
absolutely continuous, we can define the excitation function $\phi^m$
(associated with the rate at which the points in generation $m-1$ produce
offspring in generation $m$) by $a \pi^m(dt) = \phi^m(t)\, dt$.

More precisely, the \define{nonhomogeneous single-progenitor Hawkes process
$\tH$ with parameters $a$ and $\pi=\seqm{\upm{\pi}}$} is defined by
\begin{align}
\label{equ:rec-Hawkes-union}
\tH:=
\cup_{m\in\N_0} \utH{m},
\end{align}
where the sequence $\seqmz{\tH^{m}}$ of "generations" is built
from a double sequence $\upm{\tP}(k)$, $m\in\N$, $k\in\N_0$
of independent Poisson processes, where $\upm{\tP}(k)$ has intensity
$\upm{a \pi}$, for each $k\in\N_0$.
The zero-th generation $\tH^{0}$
is simply the Dirac mass $\dl_0$ at 0, i.e., a deterministic point process
with a single point at $0$, representing the lone progenitor.
Once the first $m$ generations $\utH{0} , \dots , \utH{m-1}$,
$m\in\N$,  have
been constructed, we set
\begin{align}
\label{equ:rec-Hawkes}
\utH{m} := \bigcup_{k\in \N_0}
\sets{\up{T}{m-1}(k)+S}{ S\in \utP{m-1}(k)} 
\end{align}
where $\seqkz{ \up{T}{m-1}(k)}$ denotes the
point sequence of $\utH{m-1}$.
In keeping with the convention introduced above, the first union
is taken over $k$ such that $\up{T}{m-1}(k)<\infty$.

In the sequel,
we often identify a point process with its (random) point set.
Moreover, we abuse the notation and write, for example,
$\utP{m-1}(T)$ for the
Poisson process $\utP{m-1}(k)$
whose index $k$ is such that $T=T^{m-1}(k)$.
This way, \eqref{equ:rec-Hawkes} takes a more legible
form
\[
\utH{m} =
\bigcup_{T\in \utH{m-1}} \prn{T+\utP{m-1}(T)}.
\]

The Hawkes process is defined as a superposition of independent
single-progenitor Hawkes processes, started at points of an
underlying Poisson process. More precisely, in addition to the
parameters $a$ and $\pi=\seqm{\upi{m}}$ of a single-progenitor process,
let a \emph{background intensity} measure $\mu\in\smlf\pos$
be given. The \define{Hawkes process with
parameters $\mu$, $a$ and $\pi$} is defined by
\begin{align}
\label{equ:Hawkes}
H := \bigcup_{T\in P} \prn{ T+\tH(T) },
\end{align}
where $P$ is a Poisson process with intensity $\mu$ and $\seqj{\tH(j)}$ is a
sequence of independent single-progenitor Hawkes processes with parameters
$(a,\seqm{\upi{m}})$,
independent of $P$. Thanks to the local finiteness of the Poisson process,
and the finiteness of single-progenitor Hawkes
processes (guaranteed by the assumption that $a<1$), $H$ is also a locally
finite random measure.

\subsection{A convergence theorem}

To describe the setting of our theorem, we consider a sequence $\seq{(a_n,
\seqm{\pi^m_n}, \mu_n)}$ of parameter triplets of nonhomogeneous Hawkes
processes. In particular, for each $n\in\N$, $a_n\in (0,1)$, $\mu_n$ is a
locally bounded measure on $\pos$, and $\seqm{\pi^m_n}$ is a sequence of
probability measures on $(0,\infty)$. For each $n$, we denote a Hawkes process
with parameters $(a_n, \seqm{\pi^m_n}, \mu_n)$ by $H_n$, and the associated
single-progenitor process by $\tH_n$. An important role will be played by the
sequence $\seq{\rho_n}$ of probability measures given by
\begin{align}
  \label{equ:rho-n-def} \rho_n := (1-a_n) \Big( \delta_0 + \sum_{k=1}^{\infty}
  (a_n)^k ( \pi^1_n * \dots * \pi^k_n) \Big) \in\smp\pos \efor n \in \N.
\end{align} 
The measure  
$\rho_n$ is the mean measure of $(1-a_n) \tH_n$, i.e., 
$\rho_n[B]/(1-a_n)$ is the 
expected number of points of $\tH_n$ in $B$
for any $B \in \sB\pos$.

For an infinitely divisible $\nu \in \smp\pos$
and $\tau > 0$, $\nu^\tau$ denotes the $\tau$-th convolutional 
power of $\nu$, i.e., 
the probability measure on $\pos$ 
characterized by $\widehat{\nu^\tau} = (\widehat{\nu})^\tau$.  
We introduce the shorthand
\begin{align}
  \label{catdef}
  \cat = \ceil{\tau (1-a_n)^{-1}}
\end{align}
and say that the array $\arr{\pi^m_n}$ is \define{locally infinitesimal} if
\begin{align}
\label{locinf}
\lim_{n \to \infty} \sup_{m \leq \cat} \pi^m_n[\dl,\infty)
= 0 \eforall \tau, \dl >0.
\end{align}

Assumptions \ref{A} and \ref{A'} below, shown to be equivalent in Proposition
\ref{suine}, provide sufficient structure on the array $\arr{\pi^m_n}$ for
our main theorem to hold. 
\begin{assumption}{A}[Aggregation at scale $(1-a_n)^{-1}$] \label{A}
  There exists an infinitely divisible probability measure $\nu$ on
  $\pos$ with $\nu(\set{0})=0$ such that
\begin{align}
  \label{equ:homog}
\pi^1_n * \dots * \pi^{ \cat}_n
\tol{w} \nu^{\tau}
  \text{ as $n\to\infty$ for all $\tau >0$.}
\end{align}
\end{assumption}
\begin{remark}
\label{lenad}
Note that Assumption \ref{A} can be  reformulated using a 
row-wise independent array $\arr{X^m_n}$ of random variables 
such that $X^m_n \sim \pi^m_n$ as follows:
\begin{align}
  \label{SnL}
S_n (\tau) :=
\sum_{m=1}^{\cat} X^m_n \Rightarrow L_{\tau} \eforall \tau>0, 
\end{align}
where 
$L$ is a L\' evy subordinator with $L_1 \sim \nu$.
\end{remark}
\begin{assumption}{A'}\label{A'}
There exists 
a measure $\eta$ on $(0,\infty)$ with 
$\int_{(0,\infty)} \min(1, t)\, \eta(dt)<\infty$, 
a constant $t_0 >0$ with
$\eta(\set{t_0})=0$, and a constant $d \geq 0$ such that
\begin{enumerate}
\item $\arr{\pi^m_n}$ is locally infinitesimal,
\item $\tfrac{1}{\tau} \sum_{m=1}^{\cat} \int_0^{t_0} t\, \pi^m_n(dt) \to d +
    \int_{(0,t_0]} t\, \eta(dt)$ for all $\tau >0$, 
  \item $\tfrac{1}{\tau} \sum_{m=1}^{\cat} \pi^m_n \tol{v} \eta$ on
    $(0,\infty)$ for all
    $\tau>0$, and
\item either $d>0$ or $\eta(0,\infty) = \infty$.
\end{enumerate}
The infinitely divisible
distribution with the Laplace exponent 
$\ld \mapsto \ld d + \int_0^{\infty} (1-e^{-\ld t})\, \eta(dt)$ 
is then denoted by $\nu$. 
\end{assumption}

\begin{theorem}
\label{thm:main}
Let $\seq{H_n}$ be a sequence of Hawkes processes with parameter 
triplets
$\{ (a_n, \seqm{\pi^m_n},$ $\mu_n)\}_{n\in\N}$ with $a_n \upto 1$,  such that
(equivalent) Assumptions \ref{A} and $\ref{A'}$ hold and
 $(1-a_n) \mu_n \tol{v} \mu$ for some $\mu \in \smlf$.

Then there exists a locally finite random measure $\xi$ on $\pos$
such that
\begin{align}
\label{equ:conv-to-xi}
(1-a_n)^2  H_n \Rightarrow \xi.
\end{align}
Moreover, 
$\rho_n \tol{w} \rho$, where $\hrho = (1-\log (\hnu))^{-1}$ with $\nu$ 
as in Assumption \ref{A}, and 
the law of $\xi$ is characterized by 
\begin{align}
\label{equ:char-xi}
\bee{ \Exp{ f*\xi}} =
\Exp{  h[f] * \mu}  \eforall f \in C\pos
\ewith f(0) =0 \eand \abs{f} \leq 1/2,
\end{align}
where $h[f]$ denotes the unique locally bounded
solution of the convolutional Riccati equation
\begin{align}
\label{equ:h-eq}
h = (f+\tot h^2)*\rho.
\end{align}
\end{theorem}

While the proof of Theorem \ref{thm:main} is deferred to Appendix
\ref{sec:proof-main}, we comment on its assumptions, scope, and repercussions in
the next three subsections.

\subsection{Equivalence of and sufficient conditions for Assumptions \ref{A} and \ref{A'}} \label{sec:A} 
The equivalence of Assumptions \ref{A} and \ref{A'} follows essentially 
from the classical limit theorems for triangular arrays 
going back to Kolmogorov and his co-authors. As with 
all other proofs in this section, the 
remaining details and references that complete the proof of Proposition \ref{suine} are deferred to  
Appendix \ref{sec:proof-asm}. 
\begin{proposition}
\label{suine}
Assumptions \ref{A} and \ref{A'} are equivalent.
\end{proposition}
In view of this equivalence, we focus 
solely on Assumption \ref{A} in the sequel and 
compare it  with two natural weaker conditions, formulated as 
Assumptions \ref{B1} and \ref{B2}, below. The first one 
is stated in terms of the 
mean measures $\seq{\rho_n}$ and not directly in terms of the excitation 
profiles $\arr{\pi^m_n}$. It is  implied by Assumption \ref{A} (via Theorem
\ref{thm:main}) and is
clearly necessary for Theorem \ref{thm:main} to hold.  
\begin{assumption}{B1}[Convergence of mean measures]
  \label{B1}
  There exists $\rho \in \smp\pos$ such that $\rho(\set{0})=0$ and $\rho_n \tol{w} \rho$.
\end{assumption}
The convergence of the sequence $\seq{\rho_n}$ is equivalent to the
convergence in distribution of the random sum
$\sum_{m=1}^{G(1-a_n)} X^m_n$, where $\arr{X^m_n}$ is as in Remark
\ref{lenad} and  
$G(p)$ is an independent geometrically distributed random variable
with parameter $p$. 

\smallskip

The second, weaker, condition is Assumption \ref{A} specialized to $\tau=1$.  
\begin{assumption}{B2}[Central Limit Theorem]
  \label{B2}
  There exists a probability measure $\nu$ on $\pos$ such that $\nu(\set{0})=0$ and 
  \begin{align*}
    \pi_n^1 * \dots * \pi_n^{\ceil{1/(1-a_n)}} \tol{w} \nu.
  \end{align*}
\end{assumption}

It is clear that Assumption \ref{B2} does not imply Assumption \ref{A}.
As the following simple example demonstrates, 
neither does Assumption \ref{B1}, even under the additional assumption 
of infinitesimality. 
\begin{example}
\label{beeth}
Let $a_n = 1 - 1/n$, $p_n = n^{-1/2}$, $q_n = 1- p_n$,  
$m_n = \ceil{\sqrt{n}}$ and define
  \begin{align*}
  \pi^m_n = 
  \begin{cases}
    p_n \delta_1 +  q_n \delta_0, & n - \sqrt{n}  < m \leq n  \text{
    and $n$ is even,} \\
    p_n \delta_1 + q_n \delta_0, & n  < m
    \leq n + \sqrt{n}  \text{ and $n$ is odd,} \\
    \delta_0, & \text{otherwise.}
\end{cases}
  \end{align*}
  Since $p_n \to 0$, $(\pi^n_m)$ is an infinitesimal array. 
  The convolution $\pi^1_n*\dots *\pi_n^{n}$ equals $\delta_0$ 
  when $n$ is odd, and corresponds to the binomial
  distribution with $1/p_n + o(1)$ trials and success probability $p_n$
  when $n$ is even. 
  Consequently, it does not converge
  weakly as it approaches 
  a Poisson distribution with parameter $1$ 
  along even indices, and $\delta_0$ along odd ones.

  On the other hand, it is not difficult to see that $\rho_n \tol{w} \rho$,
  where $\rho$ is a mixture of $\delta_0$ and the Poisson distribution 
  with parameter $1$, with weights $1-e^{-1}$ and $e^{-1}$. 
  The additional requirement $\rho(\set{0})=0$ can be met by
  redefining $\pi^m_n$ from $\delta_0$ to  $\delta_{p_n}$ for $m  \leq
  \sqrt{n}$, which will simply shift all distributions above by $1$. 

\end{example}

To bridge the gap between Assumptions \ref{B1} and \ref{B2} and Assumption
\ref{A}, we introduce another structural assumption requiring
approximate periodicity 
with period $d_n = o((1-a_n)^{-1})$, with an average error of size
$o(1-a_n)$. Before we state it, 
we recall that the \define{Fortet-Mourier metric} $d_{FM}$ on $\smp\pos$ is given by
\begin{align}
\label{coved}
  d_{FM}(\nu_1, \nu_2) = \sup_f \abs{ \int f \,d\nu_1  - \int f\, d\nu_2}
\end{align}
where the supremum is taken over all $[-1,1]$-valued
$1$-Lipschitz 
functions on $\pos$. 

\begin{assumption}{C}[Approximate periodicity]
  \label{C}
  The array $\arr{\pi^m_n}$ is locally infinitesimal and 
there exists a sequence $\seq{d_n}$ in $\N$ such that $(1-a_n) d_n \to 0$ and
\begin{align*}
  \lim_{n \to \infty} \sum_{m=1}^{\cat} d_{FM}(\pi^{m}_n, \pi_n^{m \tmod{d_n}
  })=0 \eforall \tau>0, 
\end{align*}
where $m \tmod{d_n}$ denotes the remainder, taken in $\set{1,\dots, d_n}$,
after division by $d_n$. 
\end{assumption}

We have the following implications between the assumptions above:
\begin{proposition} \ 
\label{Todea}
      Either Assumption \ref{B1} or Assumption \ref{B2}, together 
      with Assumption \ref{C}, implies Assumption
      \ref{A}. 
      Moreover, the probability measure $\rho$ of Assumption \ref{B1}
      and $\nu$ of Assumption \ref{A} are related by 
      $\hrho=(1-\log(\hnu))^{-1}$. 
\end{proposition}

\section{Properties of Feller random measures}
\label{sec:Feller}
Measures $\mu\in \smlf\pos$ and 
$\rho \in \smp\pos$ with $\rho(\set{0}) = 0$ 
that appear in Theorem \ref{thm:main}
uniquely define
a locally finite random measure $\xi$ via 
\eqref{equ:char-xi}
and \eqref{equ:h-eq}.
We call such $\xi$ the \define{Feller random measure} with
parameters $\mu$ and $\rho$, and denote this by $\xi \sim F(\mu,\rho)$.
When it exists,
a nonnegative measurable process $\{Y_t\}_{t\geq 0}$, such that
$\xi[A] = \int_A Y_t\, dt \eforall A \subseteq \sB\pos$ a.s.~
is called the \define{density} of $\xi$.
\begin{remark}
\label{rem:carls}
It has been shown in
\cite{JaiRos15} and \cite{JaiRos16} that, when $\rho$ is the
Mittag-Leffler distribution with index $\alpha>1/2$ and $\mu$ is the
Lebesgue measure on $\pos$, the
Feller random measure $\xi$
admits a density $Y$
which has the same distribution as a solution to a
Volterra-type stochastic differential
equation of the form
\begin{align}
\label{equ:unmet}
Y_t = Y_0 + c_1\int_0^t (t-s)^{\al-1} (\theta - Y_s)\, ds + c_2
\int_0^t (t-s)^{\al-1} \sqrt{Y_s}\, dB_s,
\end{align}
where $c_1, c_2$ and $\theta$ are constants and $B$ is a Brownian motion.
The form of \eqref{equ:unmet} explains why $Y$ is called the
fractional (or rough) CIR (or Feller) process in the literature, and
why we adopt
the name Feller random measure for the general case. In addition to
\cite{JaiRos15, JaiRos16}, we refer the reader to \cite{ElERos19}
for further information on fractional CIR process and to
\cite{AbiLarPul19} for a treatment of more general stochastic
differential equations of the Volterra type.
\end{remark}

\subsection{Attainable limiting distributions \texorpdfstring{$\rho$}{rho}}
As we have seen above, a Feller random measure $F(\mu,\rho)$ can be 
constructed from measures
$\mu \in \smlf\pos$ and $\rho \in \smp\pos$, as long as they arise as limits in
Theorem \ref{thm:main}. While this clearly imposes no restriction on $\mu$, the
situation for $\rho \in \smp\pos$ is more subtle.
\begin{definition}
\label{def:gid}
A random variable $X$ is said to be
\define{geometrically infinitely divisible (GID)} if
for each $p\in (0,1)$ there exists an i.i.d.~sequence $\seqm{X_m(p)}$ of
random variables such that
\[ X \eqd \sum_{m=1}^{G(p)} X_m(p),\] where $G(p)$ is
an $\N_0$-valued geometrically distributed random variable with
parameter (probability of success) $p$, independent of
the sequence $\seqm{X_m(p)}$.
\end{definition}
This notion was introduced in \cite{KleManMel84} as  part of the
answer to the following question posed by V.~M.~Zolotarev:
characterize the family $\sY$ of distributions of random
variables $Y$ such that,
for any $p\in (0,1)$, there exists a random variable $X(p)$ such that
\begin{align*}
Y \eqd X(p)  + B(p) Y
\end{align*}
where $Y$, $X(p)$ and $B(p)$ are independent and
$\pp{B(p)=1} = 1- \pp{B(p)=0} = p$. 
In the same paper, the authors show that
$\sY$ coincides with the set of all GID
distributions. Furthermore, they prove
that
a probability measure $\rho$ on $\pos$ is GID if and only if its
Laplace transform $\hrho$ has the form
\begin{align}
\label{equ:GID-Lap}
\hrho(\ld) = \frac{1}{1 - \log(\hnu(\ld))},
\end{align}
where $\hnu$ is the Laplace transform of some infinitely divisible
probability measure $\nu$ on $\pos$.

By the L\' evy-Khinchin representation, this is equivalent to
$\hrho$ admitting the following  form
\begin{align}
\label{equ:GLH}
\hat{\rho}(\ld) = \frac{1}{1+ d \ld + \int_0^{\infty} (1-e^{-\ld t})\,
\eta(dt)}\,,
\end{align}
for some
constant $d\geq 0$ and some measure $\eta$ on $(0,\infty)$ with
$\int_0^{\infty} \min(1,t) \, \eta(dt) < \infty$.

The very definition of infinite divisibility implies that for any sequence $\seq{a_n}$ with $a_n \upto 1$ and
any infinitely divisible probability measure $\nu$ on $\pos$ there exists a
row-constant array, i.e., an array
of the form $\pi^m_n = \pi_n$,  such that  Assumption \ref{B2} holds.
In this case, Assumption \ref{C} holds trivially, and  Theorem  
\ref{thm:main} yields the following result (for related work 
in the context of the \emph{theory of random summation} the reader is
referred to \cite{KorKru93}):
\begin{proposition}
\label{error}
Any probability measure $\rho$ that appears as the limit 
$\rho = \lim \rho_n$
in Theorem \ref{thm:main} is GID. Conversely, given any sequence
$\seq{a_n}$ with $a_n \upto 1$ and any GID probability $\rho$, there exists
a locally infinitesimal, row-constant array
$\arr{\pi^m_n} = \arr{\pi_n}$ of excitation profiles so that 
$\rho = \lim \rho_n$. 
\end{proposition}
Although general arrays $\arr{\pi^m_n}$ provide additional 
modeling flexibility, 
row-constant arrays, as stated in Proposition \ref{error},  
already generate all possible
limiting probabilities $\rho$. However, when the dependence on $n$
is further restricted, the class of attainable probability 
measures $\rho$ becomes more limited. In the case
when $\pi^m_n=\pi_n$ is given by a rescaled version of the same probability
distribution $\pi$ (as is the case, e.g., in \cite{JaiRos15, JaiRos16}),
the limiting distribution, up to scaling, belongs to a specific
one-parameter family, which in turn essentially determines the sequence 
$\seq{a_n}$.
\begin{definition}
\label{def:Solon}
A probability measure $\rho$ on $\pos$ is called the
\define{Mittag-Leffler} distribution with parameter
$\alpha \in (0,1]$ if its Laplace
transform $\hrho$ takes the form
\begin{align}
\label{equ:psi-ML}
\hrho(\ld) = \frac{1}{1+\ld^{\al}}.
\end{align}
\end{definition}

The Mittag-Leffler distribution admits an explicit density
\begin{align}
\label{equ:ML-dens}
p^{\al}(t) =  t^{\al-1} E_{\al,\al}(-t^{\al}) \efor t\geq 0,
\end{align}
where, for $\alpha,\beta>0$,  the
\define{Mittag-Leffler function} $E_{\al,\be}$ is given by
\begin{align*}
E_{\al,\be}(x) = \sum_{n=0}^{\infty} \frac{x^n}{\Gamma(\al n + \be)}
\efor x\geq 0,
\end{align*}
and $\Gamma$ denotes the Gamma-function.
In the special case $\al=1$ this distribution is exponential, with
rate $1$,
while for $\al=1/2$ its density takes an especially simple form.
Indeed, a straightforward, albeit somewhat tedious, derivation directly from
the definition
yields
\begin{align}
\label{equ:p-half}
p^{1/2}(t) &= \sqrt{\frac{2}{\pi} }
\prn{\Big. \sqrt{2t} - m\prn{ \oo{\sqrt{2t}}}},
\end{align}
where $m(x) =  \frac{1-\Phi(x)}{\vp(x)}$ is the ratio (known as the Mills
ratio) of the survival function $1-\Phi$
and the density $\vp$ of the standard normal distribution. 

For the convenience of the reader, we present here
a classical result (see \cite[Theorem 2.5.1, p.~35 and Theorem 2.5.2,
p.~37]{GneKor96}) reformulated to fit our setting:
\begin{proposition}
\label{pro:Digor}
Let $\pi$ be a probability measure on $\pos$ with $\pi(\set{0})=0$
and let the array $\arr{\pi^m_n}$ be given by
\begin{align}
\label{equ:pin-phi}
\pi^m_n (B) = \pi\prn{ n B } \eforall B\in \sB\pos.
\end{align}
Then $\seqmn{\pi^m_n}$ satisfies Assumption \ref{A}
if and only if one of the following two conditions is satisfied:
\begin{enumerate}
\item  \label{ite:exp-case}
  $\lim_{t\to\infty} \frac{\pi[t,\infty)}{\oo{t} \int_0^t \pi[s,\infty)\, ds}
  = 0$, or
\item \label{ite:other-case}
  there exists $\al\in (0,1)$ such that
  $\lim_{t\to\infty} \frac{\pi[t,\infty)}{\pi[c t,\infty)} = c^{\al}$ for all
  $c> 0$.
\end{enumerate}
In either case, if $\seq{a_n}$ is such that 
$a_n  =1-\kappa_n n^{-\al}$ for some $\seq{\kappa_n}$ with $
\kappa_n \to \kappa \in (0,\infty)$, then $\seq{\rho_n}$ converges to  
a (possibly scaled) Mittag-Leffler distribution with parameter $\al$
(where $\al=1$ in case \ref{ite:exp-case}).
\end{proposition}

\begin{remark} \
\label{rem:crape}
\begin{enumerate}
\item In the case \ref{ite:exp-case} the limiting distribution
  $\rho$ is exponential and the condition is satisfied,
  in particular, if the probability
  measure $\pi$ admits a finite first moment $\int_0^{\infty} t\,
  \pi(dt)$ as in \cite{JaiRos15}.
\item The case \ref{ite:other-case} covers all $\pi$ such that
  $\pi[t,\infty)$ is a regularly varying function with a nontrivial tail,
  i.e.,
  \begin{align*}
    \pi[t,\infty) \sim l(t) t^{-\al} \text{ as } t\to\infty
  \end{align*}
  for some $\al\in (0,1)$ and some
  slowly varying (e.g.,~constant) function $l$.
\item
  The choice of $n$ as the scaling factor in \eqref{equ:pin-phi} is
  simply a convenient normalization and can be generalized easily.
\end{enumerate}
\end{remark}

\subsection{Distributional properties}
Proposition \ref{pro:agist} below provides a recursive scheme for
efficient computation of cumulants (and, therefore,  moments) of
integral functionals of the Feller random measure $\xi$.
It is based on an the Adomian decomposition of the convolutional
Riccati equation \eqref{equ:h-eq} which
is described in Appendix \ref{sec:aux-Hawkes},
together with several other
properties of solutions of convolutional Riccati equations used in the proof
of the main theorem.

\subsubsection{Cumulants}
We recall that a real sequence $\seq{\kappa_n[Y]}$ is called the
\define{sequence of cumulants} of (the distribution of) the random variable $Y$
if
\begin{align*}
\ee{\exp{ \eps Y}} = \exp{\textstyle\sum_{n\geq 1} \frac{\eps^n}{n!}
\kappa_n[Y]}
\end{align*}
for $\eps$ in some neighborhood of $0$. For a pair $(X,Y)$, we also
define the \define{partial cumulants} $\kappa_n[X,Y]$ by
\begin{align*}
\ee{\exp{ X+\eps Y}} = \exp{\textstyle\sum_{n\geq 0} \frac{\eps^n}{n!}\,
\kappa_n[X,Y] },
\end{align*}
provided the series converges for $\eps$ in some neighborhood of $0$.
As is well known, cumulants determine the moments of the 
distribution. 
Indeed, the two sequences
are related to one another via an explicit formula based on Fa\`a di
Bruno's formula and involving Bell polynomials (see \cite{Smi95}).

\smallskip 
The next result shows that cumulants of random variables of the form
$f*\xi$, where $\xi \sim F(\mu,\rho)$, satisfy a simple
recursion and admit explicit representation. Partial
cumulants can be expressed in terms of a solution of a system of
convolutional equations. The proof is provided in Appendix \ref{sec:Ado}. 
\begin{proposition} Let $\xi$ be a Feller random measure with
parameters $\mu$ and $\rho$
and let $f,f_0:\pos \to \R$ be continuous functions which vanish at $0$.
\label{pro:agist}
\begin{enumerate}
\item  \label{ite:cumulants}
  The cumulants $\kappa_n[ f*\xi]$ of $f*\xi$ are given by
  \[ \kappa_n[f*\xi] = n!\, K_n*\mu \efor n\geq 1, \]
  where the functions $K_n \in \ssil\pos$ are defined recursively by
  \begin{align}
    \label{equ:cocco}
    K_1= f*\rho , \ K_n = \tot \prn{ \tsum_{i=1}^{n-1} K_i K_{n-i}}*\rho
    \efor n\geq 2.
  \end{align}
\item \label{ite:partial}
  If $\sup f_0 < 1/2$, the partial cumulants $\kappa_n[f_0*\xi,
  f*\xi]$ are given by
  \[ \kappa_n[f_0*\xi, f*\xi] = n!\, K'_n*\mu \efor n\geq 1, \]
  where $\seqz{K'_n}$ is the unique solution in
  $(\ssil\pos)^{\N_0}$ of the system
  \begin{equation}
    \label{equ:pre-system}
    \begin{split}
      K'_0 &= h[f_0], \quad \\
      K'_1 &= f+ \prn{  K'_0 K'_1}*\rho, \eand \\
      K'_n &=  \tot \prn{ \tsum_{i=0}^{n}    K'_i  K'_{n-i}}*\rho,
      \efor n \geq 2.
    \end{split}
  \end{equation}
\end{enumerate}
\end{proposition}
\begin{remark}
The first three cumulants/moments are given below:
\begin{equation}
\label{equ:cumulants}
\begin{split}
\kappa_1 & =\ee{ f*\xi } = (f*\rho)*\mu \\
\kappa_2 & =\var{ f*\xi } = \prn{\Big. (f*\rho)^2*\rho }*\mu \\
\kappa_3 & = \ee{ \prn{ f*\xi - \ee{f*\xi}}^3 } = 3 \prn{\Big. \prn{\Big.
\prn{\big. (f*\rho) ((f*\rho)^2*\rho) } }*\rho }*\mu
\end{split}
\end{equation}
\end{remark}
\subsubsection{Infinite divisibility}
The following property of $F(\mu,\rho)$ follows directly
from its characterization \eqref{equ:char-xi}, \eqref{equ:h-eq}.
\begin{proposition}
\label{pro:wanly}
Suppose that $\xi_1 \sim F(\mu_1,\rho)$ and $\xi_2 \sim F(\mu_2,\rho)$ where
$\mu_1, \mu_2$ are locally finite and $\rho$ is a
probability measure with $\rho(\set{0})=0$. If $\xi_1$ and $\xi_2$
are independent, then
\[ \xi_1+\xi_2 \sim F(\mu_1+\mu_2, \rho).\]
\end{proposition}
\begin{corollary}
\label{cor:moral}
Given $\xi \in F(\mu,\rho)$,
the random variable $(f*\xi)(t)$ is infinitely divisible for all 
$t\geq 0$ and 
all $f\in C\pos$ with $f(0)=0$.
\end{corollary}
A stochastic process $\fml{Y_t}{t\geq 0}$ is said to be
\define{infinitely divisible} if all of its finite-dimensional distributions
are infinitely divisible. 
\begin{corollary}
Suppose that $F(\oo{N} \mu, \rho)$ admits a right-continuous density
$Y^{(N)}$ for each $N\in\N$. Then $Y^{(1)}$ is infinitely divisible.
\end{corollary}

\subsubsection{The Covariance Structure}
\label{sss:cov}
The polarization identity and
the expression for $\kappa_2$ in \eqref{equ:cumulants} yield
\begin{align}
\label{equ:tromp}
\Cov[ f*\xi, g*\xi] = \Bprn{\bprn{ (f*\rho) (g*\rho) }* \rho}*\mu,
\end{align}
for all $f,g \in C\pos$ with $f(0)=g(0)=0$. 
Equation \eqref{equ:tromp} can be rewritten as 
\begin{align}
\label{equ:traks}
\Cov[ f*\xi, g*\xi](t) = \iint f(t-r) g(t-s) \gamma(dr, ds),
\end{align}
where
\begin{align}
\label{Chama}
\gamma(dr,ds) = \int \rho(dr-u) \rho(ds-u) (\rho*\mu)(du) 
\end{align}
i.e.,
$\gamma(B) = \int \int \int \ind{B+(u,u)}(s,r)  \rho(ds)
\rho(dr)\,  (\rho*\mu)(du)$ for each Borel $B \subseteq \pos\times\pos$. 
In the special case when $\rho$ admits a density $p$
with respect to Lebesgue measure, the measure $\gm$ is absolutely
continuous and
\begin{align}
\label{equ:lux}
\gamma(dr,ds) = \Sigma(r,s)\, dr\, ds \ewhere
\Sigma(r,s) =
\int p(r-u) p(s-u) (p*\mu)(u)\, du.
\end{align}
A further specialization yields tight asymptotics
around the diagonal $r=s$.
For functions $f:D_f\subseteq \R^d \to \R$ and
$g:D_g \subseteq
\R^d \to \R$, we write $f \apc g$ if for each bounded $B\in\sB(\R^d)$
there exists a strictly positive constant $C$  such that
$f \leq C g$ and $g \leq C f$ on $D_f \cap D_g \cap B$.

\begin{proposition}
\label{pro:diag-asy}
Suppose that $\rho$ is a Mittag-Leffler distribution with parameter
$\al \in (0,1]$,
and that  $\mu$ is the Lebesgue measure on $\pos$.
Then
\begin{align}
\label{equ:Sig-bds}
\Sigma(r,s) \apc B_{\al}(r,s)
\end{align}
where $B_{\al}:\sets{ (s,r) \in (0,\infty)^2}{ s\ne r} \to (0,\infty)$ is
a symmetric function defined for $r<s$ by
\begin{align}
\label{equ:B-al}
B_{\al}(r,s) =
r^{2\al} s^{\al-1}
\begin{cases}
  \prn{ 1- \frac{r}{s} }^{2\al -1}, & \al<\tot,\\
  1-\log(1-\frac{r}{s}), & \al = \tot, \\
  1, & \al > \tot.
\end{cases}
\end{align}
\end{proposition}
\begin{proof}
Being entire, the Mittag-Leffler function $E_{\al,\al}$
satisfies
$E_{\al,\al} \apc 1$, and so, thanks to \eqref{equ:ML-dens},
we have $p(t) \apc t^{\al-1}$.
Moreover, since $\mu$ is the Lebesgue measure, we have
$(p*\mu)(t) \apc t^{\al}$ so that
\begin{align*}
p(r-u) p(s-u) (p*\mu)(u)  \apc
(r-u)^{\al-1}(s-u)^{\al-1} u^{\al}.
\end{align*}
Therefore, by \eqref{equ:lux},
\begin{equation}
\label{detur}
  \begin{split}
  \Sigma(r,s) &\apc   \int_0^r (r-u)^{\al-1} (s-u)^{\al-1} u^{\al}\, du
           \\ &= r^{2 \al} s^{\al-1}
\int_0^1 (1 - w)^{\al-1}
(1 - \tfrac{r}{s}  w)^{\al-1} w^{\al} dw.
\end{split}
\end{equation}
According to \cite[eq.~(15.6.1)]{DLMF}, we have
\begin{align}
\label{tengu}
\int_0^1 (1 - w)^{\al-1} (1 - \tfrac{r}{s}  w)^{\al-1} w^{\al} dw =
\Gamma(\al+1) \Gamma(\al)\,
\hgF \prn{1-\al, \al+1; 2\al + 1; \frac{r}{s}},
\end{align}
where $\hgF$ denotes the hypergeometric function.
We have the following three 
asymptotic regimes for $\hgF$ (see \cite[\S 15.4(ii)]{DLMF}):
\begin{align*}
\hgF \prn{1-\al, \al+1; 2\al + 1; x} &\apc (1-x)^{ 2\al - 1}  &
\efor&
\al < \tot, \\
\hgF \prn{1-\al, \al+1; 2\al + 1; x} &\apc 1-\log(1-x)  & \efor
&
\al = \tot , \eand \\
\hgF \prn{1-\al, \al+1; 2\al + 1; x} &\apc 1 &  \efor&
\al > \tot, 
\end{align*}
and they 
establish \eqref{equ:Sig-bds} via \eqref{detur} and \eqref{tengu}. 
\end{proof}

\subsubsection{Absolute continuity and absence of atoms}
It is a direct consequence of Proposition \ref{pro:diag-asy}, as formulated
in Corollary
\ref{cor:gater} below,  that in the special case of the Mittag-Leffler 
measure $\rho$, a density may or may not exist, depending
on the value of the parameter $\al$.  
\begin{corollary}
\label{cor:gater}
Suppose that $\rho$ is a Mittag-Leffler distribution with parameter
$\al \in (0,1]$, and that  $\mu$ is the Lebesgue measure on $\pos$.
Then $\xi \sim F(\mu,\rho)$ admits a square-integrable density
if and only if $\al > 1/2$.
\end{corollary}
On the other hand, Feller measures are almost surely 
non-atomic for all parameter pairs:
\begin{proposition}
\label{strow}
For any GID $\rho$ with $\rho(\set{0})=0$ and any locally finite $\mu$, the
Feller random measure 
$\xi \sim F(\mu,\rho)$ is almost surely atom-free, i.e. 
\begin{align*}
  \xi(\set{t}) =0 \eforall t\geq 0 \text{ a.s.}
\end{align*}
\end{proposition}

\begin{proof}
The relation \eqref{equ:traks}  shows that
the measure $\gamma$ given by \eqref{Chama} is the covariance measure of
$\xi$, i.e., that 
\begin{align}
  \label{Mosgu}
  \ee{ \iint  \ind{C}(r,s) \, \xi(dr) \, \xi(ds) } = \gamma(C),
\text{ for all Borel sets $C \subseteq \pos^2$. }
\end{align}
In particular, when $D =\sets{ (t,t) }{ t \in \pos}$
is the diagonal
of $\pos^2$, equation \eqref{Mosgu} becomes
\begin{align*}
  \Bee{ \sum_{t \geq 0} \xi(\set{t})^2} = \gamma(D),
\end{align*}
where the sum is taken over the at most countable set of $t$ with
$\xi(\set{t})>0$. 

It follows directly from \eqref{Chama} that
\begin{align}\label{scrow}
  \gamma(D) = \int \sum_{t\geq u} \rho(\set{t-u})^2 (\rho*\mu)(du) =
  \mu\pos \sum_t \rho(\set{t})^2.
\end{align}
Let $\nu$ be the infinitely divisible probability measure on $\pos$ 
related to $\rho$ by \eqref{equ:GID-Lap}, and let $d$ and $\eta$ be the 
drift and the jump measure of $\nu$, respectively. 
Since $\rho(\set{0}) = 0$, we have $\lim_{\ld \to \infty} \hrho(\ld) = 0$,
which implies that either $d>0$ or
$\eta(0,\infty)=+\infty$. Either condition (see \cite[Theorem 
27.4, p.~175]{Sat99}) implies that $\nu^{\tau}$ is atomless for each
$\tau>0$. Being a mixture of the family
$\sets{\nu^{\tau}}{\tau > 0}$
of atomless measures, the measure $\rho$ is atomless as well, which implies that
$\gamma(D) = 0$ via \eqref{scrow}.
\end{proof}

\appendix

\section{Auxiliary results and the proof of the main theorem}
\label{sec:appendix}

\subsection{Proofs of results in subsection \ref{sec:A}}
This section contains proofs of Propositions \ref{suine} and \ref{Todea}
in Subsection \ref{sec:A}. 
\label{sec:proof-asm}
\begin{proof}[Proof of Proposition \ref{suine}]
 We only need to prove that Assumption \ref{A} implies local infinitesimality. 
The rest is classical (see, e.g.,  \cite[Theorem  7.14, p.~157]{Kal21}); 
condition $4.$ is equivalent to $\nu(\set{0})=0$ as in the last paragraph
of the proof of Proposition \ref{strow} above. 

We argue by contradiction and assume that local infinitesimality does not hold. Passing to a subsequence if necessary, 
we have
\begin{align*}
  \pp{ S_{n}(\tau_n+1/n) - S_n(\tau_n) \geq \dl } \geq \dl,
\end{align*}
where $S_n$ is given by \eqref{SnL},  
for some $\dl>0$ and some  sequence $\seq{\tau_n}$ in $[0,\tau)$ such that
$\tau_n \to \tau' \in [0,\tau]$. 
It follows that 
\begin{align}
  \label{tocn}
  ((1-\dl) + \dl e^{-\ld \dl}) \phi_n(\ld, \tau_n + 1/n) \geq
  \phi_n(\ld, \tau_n)
\ewhere \phi_n(\ld, \sg) = \ee{ e^{-\ld S_n(\sg)}},
\end{align}
for $\ld \geq 0$. 
Since the random variables $X^m_n$ are nonnegative, 
the functions $\phi_n(\ld, \cdot)$ are monotone and converge pointwise 
to the
continuous function $\tau \mapsto \hnu^{\tau}$. 
Therefore, both $\phi_n(\ld, \tau_n+1/n)$ and $\phi_n(\ld, \tau_n)$ converge
to $\hnu^{\tau'}$, which contradicts \eqref{tocn}. 
\end{proof}

\begin{proof}[Proof of Proposition \ref{Todea}]
We adopt Assumption \ref{C} and use the shorthands
$\eps_n = 1 - a_n$ and
$m_n(\tau) = \ct$ throughout,  and agree that all $o$
and $O$ statements in this proof are uniform in $\tau$ on compacts.
We also define 
\begin{align*}
  P_n(\tau) = \sum_{m=1}^{\cat} - \log (\hpi^m_n) \eand
  P'_n(\tau) = \sum_{m=1}^{\cat} - \log (\hpi^{m \tmod{d_n}}_n)
\efor \tau > 0,
\end{align*}
where the dependence on the Laplace-transform 
parameter $\ld$ is notationally suppressed throughout most of the proof. 

Since the function $t \mapsto 1/(1+\ld) e^{-\ld t}$ is $[0,1]$-valued and
$1$-Lipschitz, the definition \eqref{coved} of the Fortet-Mourier metric 
$d_{FM}$ implies
\begin{align*}
  \babs{ \hpi^m_n(\ld) - \hpi^{m \tmod{d_n}}_n(\ld)} 
  \leq (1+\ld) d_{FM} \bprn{ \hpi^m_n(\ld) , \hpi^{m \tmod{d_n}}_n(\ld)}.
\end{align*}
Thanks to the assumed local infinitesimality of
$\arr{\pi^m_n}$, we have
\begin{align*}
  \lim_{n \to \infty} \inf_{m \leq m_n(\tau)} \hpi^m_n = 1 \eforall \tau >
  0,
\end{align*}
so that the elementary inequality
$\abs{\log(x) - \log(y)} \le \abs{x-y}/\min(x,y)$, valid for $x,y>0$, 
yields the following estimate
\begin{align*}
  \babs{ \log \hpi^m_n(\ld) - \log \hpi^{m \tmod{d_n}}_n(\ld)}   \leq
  (2+\ld ) d_{FM}\Bprn{ \hpi^m_n(\ld), \hpi^{m \tmod{d_n}}_n(\ld)} \ 
\end{align*}
for all large enough $n$ and all $m \leq m_n(\tau)$. Consequently, 
\begin{align}
  \label{ppr}
  \abs{P_n(\tau) - P'_n(\tau)} = o(1),
\end{align}
A straightforward computation 
reveals that 
  \begin{align}
    \label{St}
  \abs{\tau P'_n(1) - P'_n(\tau)} \leq  O(P'_n(\eps_n d_n)).
\end{align}
Similarly,  and using \eqref{ppr} for the last equality, we get
\begin{align*}
  P_n(\eps_n d_n) = P'_n(\eps_n d_n) = \eps_n d_n O(P'_n(1)) = \eps_n d_n
  O(P_n(1)),
\end{align*}
and, from there, 
\begin{align} \label{rance}
P_n(\tau) = \tau P_n(1) + \eps_n d_n O(P_n(1)).
\end{align}
Assumption \ref{B2} states that $P_n(1) \to -
\log(\hnu)$, 
for some $\nu \in \smp\pos$.
This, in particular, implies that $P_n(1) = O(1)$, and, then, by
\eqref{rance},  that $P_n(\tau) = 
\tau P_n(1) + o(1)$. 
Hence, by \eqref{St},
\begin{align*}
  \lim_n P_n(\tau) = \lim_n \tau P_n(1) = - \tau \log( \hnu) \eforall \tau
  > 0.
\end{align*}
It follows that, for each $\tau>0$,   
$\hnu^\tau$ is a Laplace transform of a probability measure. Hence,
$\nu$ is infinitely divisible and 
Assumptions \ref{C} and \ref{B2} together imply Assumption \ref{A}. 

\medskip

Next, we impose Assumptions \ref{B1} and \ref{C}, and observe that
\begin{align*}
\hrho_n  =& \sum_{m=0}^{\infty} \eps_n e^{ - P_n(m \eps_n)} (1-\eps_n)^{m}  = 
   \int_0^{\infty} e^{-P_n(\sg)} (1-\eps_n)^{m_n(\sg)}\, d\sg.
\end{align*}
A second-order expansion of the function $\log(1-x)$ around $x=0$ gives
\begin{align}
  \label{unode}
  \Exp{-\sg} \geq (1-\eps_n)^{\cas} \geq \Exp{-\sg (1+2\eps_n) }  \eforall
  n\in\N \ewith \eps_n < 1/2 \eand  \sg \in \pos, 
\end{align}
so that 
\begin{align*}
  \hrho_n = \int_0^{\infty} e^{-\sg} e^{- P_n(\sg)}\, d\sg + o(1).
\end{align*}
Similarly, $\hrho'_n = \int_0^{\infty} e^{-\sg} e^{-P'_n(\sg)}\, d\sg +
o(1)$, where the definition of $\hrho'_n$ is the same as that of $\hrho_n$,
but with $P$ replaced by $P'$. It follows that
\begin{align*}
  \abs{\hrho'_n - \hrho_n} 
  & \leq
  \int_0^{\infty} e^{-\sg} 
  \abs{ e^{- P_n(\sg)} - e^{-P'_n(\sg)}} \, d\sg + o(1) 
    \\ &\leq \int_0^{\infty} e^{-\sg} 
    \max\prn{ \abs{P_n(\sg) - P'_n(\sg)}, 1}\, d\sg + o(1). 
\end{align*}
The estimate \eqref{ppr}
and the dominated convergence theorem imply that $\abs{\hrho'_n - \hrho_n} \to
0$. It follows from Assumption \ref{B1} that
\begin{align*}
  \hrho'_n \to \hrho := \lim_n \hrho_n. 
\end{align*}
For $m \in \N$ we define
$r^m_n = m \tmod{d_n} \in \set{1,\dots, d_n}$ and 
$q^m_n = (m - r^m_n)/ d_n $, so that 
\begin{align*}
  \hrho'_n 
  &:= \sum_{m=0}^{\infty} \eps_n (1-\eps_n)^m 
  e^{ - P'_n(m \eps_n) } 
  \\ &= \sum_{m=0}^{\infty} \eps_n 
(1-\eps_n)^{q^m_n d_n} e^{ - q^m_n P'_n(d_n\eps_n)}
(1-\eps_n)^{r^m_n} e^{ - P'_n (r^m_n\eps_n)}
  \\ &= \sum_{q=0}^{\infty} \eps_n 
(1-\eps_n)^{q d_n} e^{ -  q P'_n(d_n\eps_n)}
\sum_{r=1}^{d_n} (1-\eps_n)^{r} e^{ - P'_n (r\eps_n)}
  \\ &= \frac{R_n \eps_n d_n}{ 
  1- e^{- P'_n(d_n\eps_n)} (1-\eps_n)^{d_n}}   \ewhere
R_n = \oo{d_n} \sum_{r=1}^{d_n} e^{-P'_n(r\eps_n)} (1-\eps_n)^r.
\end{align*}
Since $\hrho_n \to \hrho \in (0,\infty)$, 
the condition $\eps_n d_n \to 0$ forces 
\begin{align*}
  \frac{R_n}{ 1-\Exp{-P'_n(\eps_n d_n)}(1-\eps_n)^{d_n}} \to \infty. 
\end{align*}
The observation that $R_n \leq 1$ together with the fact that 
$(1-\eps_n)^{d_n} = 1 - O(\eps_n d_n) \to 1$ implies that
\begin{align}
  \label{sndn}
P'_n(\eps_n d_n) \to 0. 
\end{align}
The two-sided bound
\begin{align*}
1 \geq R_n \geq e^{ - P'_n(\eps_n d_n)} (1-\eps_n)^{d_n}
\end{align*}
then yields $R_n \to 1$, and, together with \eqref{sndn}, shows that 
\begin{align*}
  \frac{P'_n(\eps_n d_n)}{\eps_n d_n} \to \oo\hrho - 1.
\end{align*}
Next, we use \eqref{ppr} for
\begin{align*}
  P_n(1) = P'_n(1) + o(1)  = 
  \eps_n d_n \left\lfloor \frac{\co}{d_n} \right\rfloor
  \frac{P'_n(\eps_n d_n)}{\eps_n d_n} + P'_n(\eps_n r_n) + o(1),
\end{align*}
for some $r_n \in \set{1,\dots, d_n}$. Since 
$0\leq P'_n(\eps_n r_n) \leq P'_n(\eps_n d_n) \to 0$ and 
$\eps_n d_n \lfloor \co/ d_n  \rfloor \to 1$, 
  we have $P_n(1) \to \hrho^{-1} - 1$. On the other hand, $P_n(1)$ is a sum
  of negative log-Laplace transforms of an infinitesimal array, so its
  limit must be a negative log transform of an infinitely divisible
  distribution (see \cite[Corollary 7.13, p.~156]{Kal21}). Therefore,
  Assumption \ref{B1} holds and so does Assumption \ref{A}, by the first
  part of the proof. 
\end{proof}

\subsection{The convolutional Riccati equation}
\label{sec:Riccati}
This section focuses on the properties of the solutions of the
convolutional Riccati equation
\begin{align}
\label{equ:dough}
K = F + \tot K^2*\rho
\end{align}
used throughout the paper. We fix $T\geq 0$ and consider functions
defined on $[0,T]$. Extensions to $\ssil\pos$
are straightforward. We assume that $\rho\in\smp\pos$ satisfies
$\rho(\set{0})=0$, but do not put any other restrictions (such as GID) on it. 

\subsubsection{A stability estimate and a comparison principle} 
The central result of this subsection is a stability estimate 
\eqref{equ:mezzo} in Proposition \ref{pro:tangi} below; 
it will be used in the proof of Theorem \ref{thm:main} in 
Section \ref{sec:proof-main}.

\begin{proposition}
\label{pro:tangi}
Suppose that $M\geq 0$ and $K_i, F_i \in \ssiT$ are such that
$\n{K_i}_{\ssiT} \leq M$ for $i=1,2$ and
\begin{align*}
K_i =  F_i + \tot K^2_i*\rho  \efor i=1,2,
\end{align*}
then
\begin{align}
\label{equ:mezzo}
\int_0^T \n{ K_2 - K_1}_{\ssit}\, dt \leq C \int_0^T
\n{F_1 - F_2}_{\ssit}\, dt.
\end{align}
where $C$ depends only on $\rho$, $M$ and $T$. 
\end{proposition}
\begin{proof}
For $t\in \zT$, we define
$m(t) = \n{K_2 - K_1}_{\ssi\zt}$ and
$m_F(t) = \n{F_2 - F_1}_{\ssi\zt}$.
For $s\leq t \leq T$, we have
\begin{align*}
  \abs{K_2(s) - K_1(s)}
&\leq \abs{F_2(s) - F_1(s)} + \tot \int_0^s \abs{K_2^2(s-u) -
K_1^2(s-u)}\, \rho(du) \\
&\leq m_F(s) + M \int_0^s m(s-u)\, \rho(du),
\end{align*}
so that
\begin{align*}
m \leq  m_F + M  (m*\rho) \text{ on } \zT.
\end{align*}
We multiply both sides by $\exp{-\ld \cdot}$ and integrate over $\zT$ to
obtain
\begin{align*}
\int_0^T e^{-\ld t} m(t)\, dt
&\leq \int_0^T e^{-\ld t} m_F(t)\, dt +
M \int_0^T  \int_0^t e^{-\ld(t-u)} m(t-u)\, e^{-\ld u} \rho(du)\, dt\\
&= \int_0^T e^{-\ld t} m_F(t)\, dt +
M \int_0^T \int_0^{T-u}  e^{-\ld s} m(s)\, ds\,
e^{-\ld u} \rho(du)
\\ & \leq \int_0^T e^{-\ld t} m_F(t)\, dt +
M \prn{\int_0^T  e^{-\ld t} m(t)\, dt} \prn{\int_0^T e^{-\ld u}
\rho(du) }
\\ & \leq \int_0^T e^{-\ld t} m_F(t)\, dt +
\prn{M \int_0^T e^{-\ld u} \rho(du) }
\int_0^T  e^{-\ld t} m(t)\, dt.
\end{align*}
Since $\rho(\set{0})=0$, there exists $\ld_0 = \ld_0(\rho, M)\geq 0$ such that
\begin{align*}
M \int_0^T e^{-\ld_0 u} \rho(du) \leq 1/2.
\end{align*}
With such $\ld_0$, we have
\[ e^{-\ld_0 T} \int_0^T m(t)\, dt
\leq \int_0^T e^{-\ld_0 t} m(t)\, dt \leq 2 \int_0^T e^{-\ld_0 t} m_F(t)\,
dt \leq 2 \int_0^T m_F(t)\, dt,\]
which implies \eqref{equ:mezzo} with the constant $2 \exp{\ld_0 T}$.
\end{proof}

Next, we restate a 
general comparison principle 
for Volterra integral equations
from \cite[Theorem 6.1, p~12]{Mil71}, 
specialized to our setting 
and adapted to our notation.  While
it is proved in \cite{Mil71} under the assumption that $\rho$ is absolutely
continuous, the proofs are readily extended to the case of a general $\rho$ with
$\rho(\set{0})=0$. When the density of $\rho$ is, additionally, restricted to
have a power-type singularity at $0$, comparison principles for the
convolutional Riccati equations are available as part of the theory of
fractional differential equations (see, e.g., \cite[Section 42.,
pp.~832]{SamKilMar93}). In the context of rough CIR processes and limits of
Hawkes processes, related comparison results include \cite[Lemma A.3, p.~38]{ElERos19} and
\cite[Section C.2, p.~35]{DurRosSzy23}.

\begin{proposition}
\label{pro:geoty}
Suppose that $F_1, F_2, K_1, K_2\in\ssiT$ are such that
\begin{align*}
K_1 \leq F_1 + \tot K_1^2*\rho \eand  K_2 \geq F_2 + \tot K_2^2*\rho.
\end{align*}
If $F_1 \leq F_2$ and $K_1+K_2 \geq 0$ then $K_1 \leq K_2$.
In particular, \eqref{equ:dough} admits at most one
solution in $\ssiT$.
\end{proposition}

\subsubsection{The Adomian decomposition}
\label{sec:Ado}
This subsection implements the Adomian decomposition, introduced in
\cite{AdoRac86}, in our particular setting. 
No original content is claimed; the results here are
routine extensions of existing results (see \cite{ElK08} or
\cite{Waz11}), but in a context that does not exactly match
anything we could find in the literature, so we include 
all proofs. 

We start from an infinite triangular system of convolutional equations:
\begin{equation}
\label{equ:system}
\begin{split}
K_0 &= B, \quad \\
K_1 &= F+ \prn{  K_0 K_1}*\rho \\
K_n &=  \prn{ \ot \sum_{i=0}^{n}    K_i  K_{n-i}}*\rho
\efor n> 2.
\end{split}
\end{equation}
in the unknown functions $\seqz{ K_n}$, where $B, F\in \ssiT$.

\begin{lemma}
\label{lem:sys-sol}
Suppose that $\n{B}_{\ssiT}<1$.
Then the system \eqref{equ:system} has a unique
solution in $(\ssiT)^{\N_0}$, denoted by
$\seqz{K_n[B,F]}$. Moreover, there exists a universal constant $C$ such that
\begin{align}
\label{equ:hen-bounds}
\n{ K_n[B,F]}_{\ssiT} \leq  C  n^{-3/2} \prn{\frac{ 2\n{F}_{\ssiT}
  }{(1-\n{B}_{\ssiT}
)^{2}}}^n  \eforall n\in\N.
\end{align}
\end{lemma}

\begin{proof}
We observe that for $n\geq 1$, the $n$-th equation in the
system \eqref{equ:system} can be written in the form
\begin{align}
\label{equ:tHn}
K_n = F_n +  \prn{B  K_n }*\rho, \ewhere
F_n
=
\begin{cases} F, & n=1, \\ \tot \sum_{i=1}^{n-1}
  (K_i K_{n-i})*\rho, & n\geq 2.
\end{cases}
\end{align}
We also observe that
$F_n$ does not involve $K_n$ or any $K_m$ with $m>n$.

Since $\n{B}_{\ssiT}<1$, it follows inductively
from a standard argument based
on Banach's fixed point theorem that the system
\eqref{equ:system} has a unique solution $\seqz{K_n[B,F]}$
in $(\ssiT)^{\N_0}$ and that its solution satisfies the following estimate
\begin{align*}
\n{ K_n[B,F]}_{\ssiT} \leq M   \n{  F_n}_{\ssiT} \ewhere
M=(1-\n{B}_{\ssiT})^{-1}.
\end{align*}
This further implies that
$\n{ K_1[B,F]}_{\ssiT} \leq M \n{F}_{\ssiT}$ and that
\begin{align}
\label{equ:hdn-Cat}
\n{ K_n[B,F]}_{\ssiT}  \leq
\tot M  \sum_{i=1}^{n-1} \n{ K_i[B,F]}_{\ssiT} \n{ K_{n-i}[B,F]}_{\ssiT} \efor
n\geq 2.
\end{align}
If $F=0$ then $K_n[B,F]=0$ for all $n\geq 1$. Otherwise, we set
\begin{align*}
c_{n} =  \frac{2^{n-1}\n{K_n[B,F]}_{\ssiT}}{ M^{2n-1} \n{F}_{\ssiT}^n}, \efor
n \geq 1,
\end{align*}
so that, by \eqref{equ:hdn-Cat},
\[ c_1 \leq 1 \eand c_n \leq \sum_{i=1}^{n-1} c_i c_{n-i} \efor n\geq
2.
\]
We recall that
the sequence $\seqz{C_n}$ of Catalan numbers satisfies
(see \cite[eq.~(1.2), p.~3]{Rom15})
the recurrence
relation
\begin{align*}
C_n = \sum_{k=1}^n C_{k-1} C_{n-k},\ C_0 = 1.
\end{align*}
Hence, by induction, $c_n \leq C_{n-1}$, for all $n \geq 1$ and so,
the standard asymptotics $
C_n \sim \pi^{-1/2} 4^{n} n^{-3/2} $
for Catalan numbers (see e.g. \cite[Theorem 3.1,
p.~15]{Rom15})
implies that $c_n \leq   C 4^n n^{-3/2}$ for some universal constant $C$,
which, in turn,  yields \eqref{equ:hen-bounds}. 
\end{proof}
\begin{lemma}
\label{lem:exp-eps}
Let $B,F\in \ssiT$ be such that
$\n{B}_{\ssiT}<1$ and $\n{F}_{\ssiT}\leq
\tot(1-\n{B}_{\ssiT})^{2}$, and let
$\seqz{K_n[B,F]} \in (\ssiT)^{\N_0}$
be the
unique solution to  the system \eqref{equ:system}. Then
the series $\tsum_{n\geq 0} K_n[B,F]$
converges absolutely in $\ssiT$ and its sum
\begin{align*}
K[B,F]:=\tsum_{n\geq  0} K_n[B,F]
\end{align*}
satisfies the equation
\begin{align}
\label{equ:frowl}
K[B,F] = F +  \tot (K[B,F])^2*\rho + B - \tot B^2*\rho
\end{align}
\end{lemma}
\begin{proof}
Let $K=K[B,F]$ and $K_n = K_n[B,F]$, for $n\in \N$. 
Lemma \ref{lem:sys-sol} and
the assumption on the size of $F$ imply that 
\[ \n{K_n}_{\ssiT} \leq C n^{-3/2} \efor n\in\N.\]
This, in turn, implies that
the series $\tsum_{n\geq 0} K_n$
converges absolutely in $\ssiT$.
Moreover,
\begin{align*}
\tot K^2*\rho &= \tot \Bprn{ \tsum_{n\geq 0} K_n }^2*\rho =
\tot K_0^2*\rho + (K_0 K_1)*\rho +
\tsum_{n\geq 2} \bprn{\tot \tsum_{i=0}^n K_i K_{n-i}}*\rho \\
&= \tot B^2 * \rho + (K_1 - F) + \tsum_{n\geq 2} K_n
= \tot B^2*\rho - B + K - F. \qedhere
\end{align*}
\end{proof}

\begin{lemma}
\label{lem:palma}
Suppose that $F\in\ssiT $ satisfies $\abs{F}\leq 1/2$ and that
$K$ solves \eqref{equ:dough}. Then
\begin{align}
\label{equ:floss}
- \n{F}_{\ssit} \leq K(t) \leq 1-\sqrt{1-2
\n{F}_{\ssit}} \eforall t\in [0,T].
\end{align}
\end{lemma}
\begin{proof}
To get the lower bound, we simply observe that $K \geq F$.
For the upper bound, we define
\begin{align*}
K_2(t) = 1-\sqrt{1 - 2\n{F}_{\ssit}}.
\end{align*}
Since $K_2$ is nonnegative and nondecreasing, we have $K_2^2*\rho
\leq K_2^2$, and, so,
\begin{align*}
K_2(t) - \tot  (K_2^2*\rho)(t)  & \geq K_2(t) - \tot K^2_2(t) =
\n{F}_{\ssit} \geq F(t)
\end{align*}
We have
\begin{align*}
(K + K_2)(t) &\geq - \n{F}_{\ssit} + 1 - \sqrt{1-2 \n{F}_{\ssit}}
= \tot \prn{ 1-\sqrt{1-2\n{F}_{\ssit}} }^2 \geq 0,
\end{align*}
which allows us to use
Proposition \ref{pro:geoty} above with $K_1 = K$ and $F_1 = F_2 = F$
to conclude that $K\leq K_2$.
\end{proof}
\begin{proposition}
\label{pro:equation}
If $\n{F}_{\ssiT}\leq 1/2$ the function
\begin{align}
\label{equ:ser}
K[F] = \tsum_{n\geq 1} K_n[0,F]
\end{align}
defines the unique solution of \eqref{equ:dough} in $\ssiT$.
Moreover, if $G\in \ssiT$ and $\eps\in\R$ are such that
$\n{F}_{\ssiT}+ \abs{\eps} \abs{G}_{\ssiT} \leq 1/2$ we have
\begin{align*}
K[F + \eps G] = K[F]+\tsum_{n\geq 1} \eps^n K_n[ K[F], G].
\end{align*}
with absolute convergence in $\ssiT$ .
\end{proposition}
\begin{proof}
In the special case $B=0$, the conditions of
Lemma \ref{lem:exp-eps} are satisfied as soon as
$\n{F}_{\ssiT}\leq 1/2$. Therefore \eqref{equ:ser}
defines a solution to \eqref{equ:dough}. Uniqueness is the
content of Proposition \ref{pro:tangi} above.

\smallskip

Let $F$, $G$ and $\eps$ be as in the second part of the statement.
Thanks to
Lemma \ref{lem:palma} above, we have
\begin{align*}
\n{K[F]}_{\ssiT} \leq \max\prn{ \n{F}_{\ssiT},  1- \sqrt{1-2
\n{F}_{\ssiT}} } = 1 -\sqrt{1-2 \n{F}_{\ssiT}},
\end{align*}
so that
\begin{align*}
\tot (1-\n{K[F]}_{\ssiT})^2 \geq 1/2-\n{F}_{\ssiT} \geq
\n{ \eps G}_{\ssiT},
\end{align*}
which is exactly what is required for Lemma \ref{lem:exp-eps} to
apply. Therefore,
\begin{align*}
K[F+\eps G] = K[F] + \tsum_{n\geq 1} K_n[ K[F],\eps G].
\end{align*}
It remains to observe that functions $\eps^{-n}
K_n[K[F], \eps G]$ solve the system \eqref{equ:system} for any
$\eps\ne 0$, so that, by uniqueness, we have
\begin{equation*} K_n[
K[F], \eps G] = \eps^n K_n[ K[F], G] \efor \eps \in \R. \qedhere
\end{equation*}
\end{proof}

\begin{proof}[Proof of Proposition \ref{pro:agist}]
To obtain \ref{ite:partial}, we simply combine the representation
of \eqref{equ:char-xi} and \eqref{equ:h-eq} with Proposition
\ref{pro:equation} above, with $F=h[f_0]$
and $G=f$. The assertion in \ref{ite:cumulants} is a special case
of \ref{ite:partial} with $F=h[0]=0$ and $G=f$. In that case, the
system \eqref{equ:pre-system} simplifies and turns into the recursive
definition given in \eqref{equ:cocco}.
\end{proof}

\subsection{Auxiliary facts about Hawkes processes}
\label{sec:aux-Hawkes}
This section introduces partial Hawkes processes and 
collects several facts 
that will be used in the proof
of Theorem \ref{thm:main} in the Section \ref{sec:proof-main}.
The notation is inherited from the
beginning of Section \ref{sec:main}.

\subsubsection{Partial Hawkes processes}
\label{sec:parH}
The parameters $a$ and $\seqm{\upm{\pi}}$ of a single-progenitor
Hawkes process can be used to construct
a double sequence of \define{partial single-progenitor Hawkes
processes}
$\tH^{[m,m+k)}$, $m\in\N_0$, $k\in\N$ which will be needed in the sequel.
The process $\tH^{ [m,m+k )}$ starts with a single individual in
generation $m\in \N_0$, and accumulates individuals over the next $k-1$ generations.
This construction is distributionally equivalent to collecting the first $k$ generations
of a single-progenitor process with parameters $a$ and
$(\pi^{m+1}, \pi^{m+2},\dots)$.
These individual generations are denoted by $\tH^{m,(m+j)}$,
$j=0,\dots, k-1$ so that
$\tH^{[m,m+k)} = \bigcup_{j=0}^{k-1} \tH^{m,(m+j)}$.

For $k\geq 1$, conditioning on the first generation $\tH^{m, (m+1)}$
of $\tH^m$
gives the following fundamental recursive distributional equality:
\begin{align}
\label{equ:tH-in-law}
\tH^{[m,m+k)} \eqd \set{0} \cup \bigcup_{T\in \utP{m+1}  }
\prn{ T+\tH^{[m+1, m+k)}(T) } \efor k\geq 1,
\end{align}
where $\utP{m+1}$ is a Poisson process with intensity
$a \up{\pi}{m+1}$, and $(\tH^{[m+1,m+k)}(T))_{T\in \tP^{m+1}}$
are independent partial single-progenitor Hawkes processes.
We accumulate over all $k$ in \eqref{equ:tH-in-law} to  obtain
\begin{align}
\label{equ:tH-in-law-lim}
\tH^{m} \eqd \set{0} \cup \bigcup_{T\in \utP{m+1}  }
\prn{ T+\tH^{m+1}(T) }.
\end{align}

The \define{convolutional moment-generating functional} $\mgf_{\xi}$
is  given by
\begin{align}
\label{equ:mgf-fxi}
\mgf_{\xi}[f] (t) = \ee{ \Exp{ (f*\xi)(t)}} \in [0,\infty]
\efor
f\in \ssil\pos, \, t \in \pos.
\end{align}
This version of the standard moment-generating
functional
proves to be easier to work with in the context of Hawkes processes
than its classical counterpart.
We note that, unlike in the standard case, the functional $\mgf_{\xi}$
depends on the parameter $t$. While this dependence does
not encode any additional information (it simply shifts the function
$f$), it leads to significantly simpler notation in the sequel.

\smallskip

Lemma \ref{lem:albee}  below provides 
the fundamental recursive relation
between convolutional moment-generating functionals of partial Hawkes processes. 
%
\begin{lemma}
\label{lem:albee}
Let $f\in \ssil\pos$ 
and $m\in\N_0$ be given. Then $ \mgf_{\tH^{[m,m)}}[f] = 1$ and
\begin{align}
\label{equ:rec-hmk}
  \mgf_{\tH^{[m,m+k)}}[f] =
  \exp{f+ a (\mgf_{\tH^{[m+1,m+k)}}[f]-1)* \upi{m+1}} \efor k\geq 1.
\end{align}
If, additionally, $f\leq a-1-\log(a)$
then
\begin{align}
\label{equ:hmk-bnd}
\mgf_{\tH^{[m,m+k)}}[f] \leq \oo{a} \eforall  k\in\N 
\quad \eand \quad
\mgf_{\tH^{m}}[f] \leq \oo{a} 
\end{align}
as well as
\begin{align}
\label{equ:hm-e}
\mgf_{\tH^m}[f] = \exp{f + a  \prn{\mgf_{\tH^{m+1}}[f]-1 }*\pi^{m+1}
}.
\end{align}
\end{lemma}
\begin{proof}
  The equality \eqref{equ:rec-hmk} follows by conditioning on the $m$-th 
  generation and 
 using the well-known fact that for a Poisson process $P$ with intensity
measure $\mu \in \smlf\pos$  and a positive $g \in \ssil\pos$ 
bounded away from $0$, we have
\begin{align}
  \label{agora}
\ee{ \tprod_{T\in P, T\leq t}  g(t-T)} = \Exp{ ((g-1)*\mu)(t) } \eforall t \geq 0.
\end{align}
The inequalities in \eqref{equ:hmk-bnd} follow inductively from \eqref{equ:rec-hmk} and 
allow 
us to pass to the limit  as $k \to \infty$ in
\eqref{equ:rec-hmk}
to obtain \eqref{equ:hm-e}. 
\end{proof}
The recursive structure of the Hawkes process can be used to derive an
expression for its moments. For $f\in \ssil\pos$, we let
\begin{align*}
\te^{[m,m+k)}[f] = \ee{ f*\tH^{[m,m+k)}} \eand \te^m[f] = \ee{ f*\tH^m},
\efor m\in\N_0, k\in\N.
\end{align*}
If we set 
$\te^{[m,m)}[f] := f$ for $m \in\N$, 
the relation \eqref{equ:tH-in-law} implies that
\begin{align}
\label{equ:equ-mmz}
\te^{[m,m+k)}[f] = f + a\,
\te^{[m+1,m+k)}*\pi^{m+1} \efor k\geq 1.
\end{align}
Therefore,
\begin{align*}
\te^{[m,m+k)}[f] = f* \prn{ \tsum_{j=0}^{k} a^j \pi^{(m,m+j]} }
\end{align*}
where
\begin{align}
\label{equ:Pimj}
\pi^{(m,m+j]} :=
\begin{cases}
\dl_0, & j=0, \\
\pi^{m+1}*\dots*\pi^{m+j}, &  j> 0.\\
\end{cases}
\end{align}
Since $a<1$ and $\ns{\pi^{(m,m+j]}}=1$ for all $m,j\in\N_0$,
we have convergence in
total variation in
\begin{align}
\label{equ:def-rhom}
\rho^m := \sum_{k=0}^{\infty} (1-a) a^k \pi^{(m,m+k]} \in \smp\pos,
\end{align}
and the following identity holds
\begin{align}
\label{equ:mm-by-rhom}
\te^m[f] = \oo{1-a} f*\rho^m \efor f\in \ssil\pos, m\in\N_0.
\end{align}

In the case of a Hawkes process $H$ with background intensity measure $\mu$,
we have the following immediate continuation of Lemma
\ref{lem:albee}. The proof uses \eqref{agora} and follows the same pattern
as the proof of Lemma \ref{lem:albee} so we omit it. 
\begin{lemma}
\label{lem:agnel}
Given $f\in \ssil\pos$ with $f\leq a-1-\log(a)$,
we have $\mgf_{H}[f]\in \ssi$ and
\begin{align}
\label{equ:sL-H-h}
\mgf_{H}[f]=\exp{\big. (\mgf_{\tH}[f]-1) * \mu }.
\end{align}
\end{lemma}

A well-known consequence of Lemma \ref{lem:agnel} above is the
following expression for the first moment
$e[f] = \ee*{f*H}$ of the Hawkes process which we record for later use:
\begin{align}
\label{equ:m1-equ}
e[f] = \te[f]*\mu = \oo{1-a} f * ( \rho *\mu ).
\end{align}
where $\rho=\rho^0$ and $\rho^0$ is given by \eqref{equ:def-rhom}.

\subsubsection{Asymptotics for the total number of points}
Let $\tH$ denote a single-progenitor Hawkes process. We are only
interested in the total number of points $\abs*{\tH}$ of $\tH$ here so
the excitation profile $\seqm{\pi^m}$ plays no role. Indeed, $\abs*{\tH}$ is the
same as the total number of points of a
Bienaym\' e-Galton-Watson process with the Poisson offspring
distribution with parameter $a$.
We start with a lemma where
$W_0:[-e^{-1},\infty) \to \R$ denotes
the principal branch of Lambert's $W$ function (see, e.g.
\cdg{4.13}{Section 4.13}).
\begin{lemma}
\label{lem:flusk}
For $\be \in \R $, let $l(\be) := \ee*{\exp*{ \be \ns{\tH}}} \in
(0,\infty]$ be the
moment-generating function of the total number $\ns{\tH}$ of points
in $\tH$.  Then
\begin{align}
\label{equ:flusk}
l(\be) =
\begin{cases} + \infty, &  \be > a-1-\log(a),\\
  -\frac{1}{a} W_0(-\exp{ \beta -a + \log(a)}) , & \be \leq  a -1 - \log(a).
\end{cases}
\end{align}
\end{lemma}
\begin{proof}
Virtually the same conditioning argument leading to Lemma
\ref{lem:albee} above can be used  to  show that
the functions
$l_k(\be) = \ee{\exp{ \be \ns{\tH^{[0,k)}}}}$, $k\in\N_0$ satisfy
\begin{align*}
l_k(\be) = \exp{ \be + a (l_{k-1}(\be)-1)} \efor k\in\N.
\end{align*}
Assuming, first, that $\be \geq 0$, the fact that
$\ns{\tH}$ is the nondecreasing limit of $\ns{\tH^{[0,k)}}$, as
$k\to\infty$ implies that $l_k(\be) \upto l(\be)$.
Since $l_0(\be)=1$,  it follows that $l(\be)$ is the smallest
fixed point above $1$,  if one exists,  of the function $F(x) = \exp{\beta + a
(x-1)}$;
otherwise, $l(\be)=+\infty$. The latter case occurs, in particular, when
$\beta>a-1-\log(a)$, as can be easily seen by inspection.
For $\beta \leq a-1-\log(a)$,
the equation $F(x)=x$ transforms into
\begin{align}
\label{equ:Lambert-equ}
(-ax) \exp{-ax}   = -a \exp{\be-a},
\end{align}
with solutions given by
\begin{align*}
x_0 = -\oo{a} W_0( - \exp{-a+\be + \log(a)}) \eand
x_{- 1} = -\oo{a} W_{- 1}( - \exp{-a+\be + \log(a)}),
\end{align*}
where $W_0$ and $W_{- 1}$ are the
two branches of the Lambert's $W$ function on $[-e^{-1},0)$.
The
principal solution $W_0$ is increasing and $W_{- 1}$ is decreasing on
$(-e^{-1},0)$, while
$W_0(-e^{-1}) = W_{- 1}(-e^{-1}) = -1$;
this is easily seen directly, but we also refer the
reader to \cdg{4.13}{Section 4.13} for a more comprehensive treatment of the
$W$ function.
It follows that the smallest solution of $F(x)=x$ above $x=1$ is given by
$x_0$ defined in \eqref{equ:Lambert-equ} above, which completes the proof of
\eqref{equ:flusk}.

The case $\beta<0$ is almost identical, with the distinction that
the sequence $\seqk{l_k(\be)}$ is
now nonincreasing and bounded from above by $1$, so we are looking
for the largest
fixed point of $F$ below $1$. Since $F(1) < 1$ in this case, $l_0(\be)=1$ is
located between the two solutions of $F(x)=x$ which leads us to choose the
principal branch $W_0$ again.
\end{proof}

\begin{proposition}
\label{pro:dhava}
Let $\seq{\tH_n}$ be a sequence of single-progenitor Hawkes processes with
parameters $\seq{a_n}$, and suppose that $\eps_n \downto 0$, where
$\eps_n = 1- a_n$.
For $\dl \in [0,1)$ we have
\begin{align}
\label{equ:tail-0}
\lim_n \oo{\eps_n} \prn{ \ee{\exp{ \eps_n^2 \drtr \ahnn}}-1} =
1-\sqrt{\dl},
\end{align}
and, for $\dl \in (0,1)$,
\begin{align}
\label{equ:tail-1}
\lim_{n}
\eps_n^{2k-1} \ee{ \ahnn^k \exp{ \eps^2_n \drtr \ahnn }}  =
\frac{(k-1)!}{2^{k-1}} \binom{2(k-1)}{k-1} \delta^{1/2-k} \efor k\in\N.
\end{align}
\end{proposition}
\begin{proof}
Let $w(x) = -W_0(-e^{-1-x^2})$ for $x \geq 0$,  where $W_0$ is the principal
branch of Lambert $W$ function.
Lemma \ref{lem:flusk} above states 
that 
\begin{align}
\label{equ:mgf-W}
\ee{ \exp{ \beta \ahnn}} = \oo{1-\eps_n} w(b_n(\be)) \ewhere b_n(\be)
=  (\beta^{max}_n - \be)^{1/2},
\end{align}
for $\be < \beta^{max}_n := - \log(1-\eps_n) - \eps_n$.
We have
\[ b_n\prn{ \tot \eps_n^2 } =
\Bprn{ -\log(1-\eps_n) - \eps_n - \tot \eps_n^2 }^{1/2}  =
O(\eps_n^{3/2}),\]
so the continuous differentiability of  $w$ on $\pos$ and the fact that
$w(0)=1$, imply that 
\[ \oo{1-\eps_n} w\prn{ b_n\prn{ \tot \eps_n^2 } } = 1 + \eps_n +
O(\eps^{3/2}_n),\]
which, in turn, yields \eqref{equ:tail-0} for $\dl=0$.
Similarly, when $\dl\in (0,1)$ we have
\[ b_n\prn{ \tot (1-\dl) \eps_n^2 } = \sqrt{\tfrac{\dl}{2}}
\eps_n  + O(\eps^2_n),\]
and the limit in \eqref{equ:tail-0}  follows from
\[ \oo{1-\eps_n} w\prn{ b_n\prn{ \tot (1-\dl) \eps_n^2 } }
= 1 + (1-\sqrt{\dl})\eps_n + o(\eps_n).\]

For the remainder of the proof, suppose that $\dl\in (0,1)$.
Standard
properties of moment-generating functions allow us to
differentiate $k\in\N$ times inside the expectation sign at each
$\beta<\beta^{max}_n$ in \eqref{equ:mgf-W} above to obtain
\begin{align}
\label{equ:Hwp0}
\ee{ \ns{\tH_n}^k
\exp{ \beta \ns{\tH_n}}} &=
\oo{1-\eps_n}
(w\circ b_n)^{(k)}(\be)
\efor \beta < \beta^{max}_n,
\end{align}
where $(\cdot)^{(k)}$ denotes the $k$-th derivative in $\beta$.
The formula of Fa\' a di Bruno states that $(w\circ b_n)^{(k)}(\be)$ admits
a representation of the form
\begin{align}
\label{equ:Nelly}
\sum \frac{k!}{m_1!\dots m_k!} w^{(m_1+\dots+m_k)}\Bprn{ b_n(\be) }
\prod_{j=1}^k \prn{
\frac{b_n^{(j)}(\be)}{j!} }^{m_j},
\end{align}
where the sum is taken over all $m_1,\dots, m_k \in \N_0$ such that
$m_1+2 m_2+\dots + k m_k = k$.

We have
$b_n^{(j)}(\be) = (-1)^j j! \binom{1/2}{j}  (\beta^{max}_n -
\be)^{1/2-j}$ so that
\begin{align*}
\prod_{j=1}^k \prn{ \frac{b_n^{(j)}(\be)}{j!}}^{m_j}
=K\, (\beta^{max}_n -
\be)^{\frac{m_1+\dots+m_k}{2} - k}
\ewhere
K=\prod_{j=1}^k \prn{(-1)^j \binom{1/2}{j}}^{m_j}.
\end{align*}
The lowest power of $(\beta^{max}_n-\be)$ appearing in
\eqref{equ:Nelly}
is $1/2-k$ and is attained precisely at
$m_1= \dots =m_{k-1}=0$, $m_k=1$.
Furthermore, all functions
$w^{(m_1+\dots+m_k)}(x)$ converge to a finite limit as $x\downto
0$ which implies that
\[ (w\circ b_n)^{(k)}(\be) =  k! w'(b_n(\be)) (-1)^k \binom{1/2}{k}
(\beta^{max}_n-\be)^{1/2-k} +
o\prn{ (\beta^{max}_n-\be)^{1/2-k}}.  \]
Since
$b_n\prn{ \tot (1-\dl) \eps_n^2 } = \sqrt{\dl/2}\, \eps_n +
o(\eps_n)$ the
fact that $\lim_{x\downto 0}
w'(x)=-\sqrt{2}$ implies that 
\[ \lim_{n}
\frac{(w\circ b_n)^{(k)}(\be_n)}{\prn{\dl \eps_n^2/2}^{1/2-k}} =
\sqrt{2}  (-1)^{k-1} k! \binom{1/2}{k},\]
which, in turn, yields \eqref{equ:tail-1}.
\end{proof}

\subsection{Proof of Theorem \ref{thm:main}}
\label{sec:proof-main}

The proof is divided into several lemmas. These lemmas depend on the results
and notation from Sections \ref{sec:Riccati} and \ref{sec:aux-Hawkes} above, as
well as on the notation and concepts introduced ahead of Theorem \ref{thm:main} in
Section \ref{sec:main}.

\begin{itemize}[label=-,leftmargin=0em, itemindent=1.5em]

  \item An arbitrary but fixed time horizon 
$T >0$ is used throughout the proof, except in the very last paragraph
where the extension to $\pos$ is discussed.  All functions, measures, etc.
are implicitly restricted to $[0,T]$ 
whenever necessary, without any change in notation. 
  \item We continue to use all 
    the notation from Subsection \ref{ssc:notation}, but we omit the 
    domain $I$ because it will almost always be $[0,T]$. Any other domain
    will be explicitly stated. We also let 
$\sL^p$ and $\slop$,  $p \in [0,\infty]$
denote the standard Lebesgue spaces on $I$;
a.e.-equal functions are not identified.
$C_0$ denotes the family of
continuous functions $f$ on $[0,T]$ with $f(0)=0$.
The family of all signed measures on $[0,T]$ is denoted by 
  $\sms$ and we write $\abs{\mu}$ for the total variation of $\mu \in \sms$
  on $[0,T]$.
\item A function $f \in \ssi$ is said to be \define{of bounded
  variation} if there exists a signed measure $Df \in \sms$  such that $f(t) =
Df[0,t]$ for all $t\in \zT$. The Hahn decomposition of $Df$ is denoted by $Df
= D^+f - D^- f$ and the total variation measure $D^+f + D^- f$, associated to
$Df$, by $\abs{D}f$. The map $\n{f}_{\bv}:= \n{ \abs{D} f}_{\sms}$ is a
Banach norm on $\bv$.
\item 
Measures and functions are assumed to vanish outside  $\zT$ so that 
limits of integration, as in $(f*\mu)(t) = \int f(t-\cdot)\, d\mu$, 
do not have to be specified. 
\item Young's inequalities $\n{h*\nu}_{\slo} \leq \n{h}_{\slo}
\n{\nu}$ and $\n{h * \nu}_{\ssi} \leq \n{h}_{\ssi} \n{\nu}$ 
are used throughout without explicit mention. 
\item We use the notation of Subsection 
  \ref{sec:parH} with an 
additional subscript $n$ indicating
membership in a sequence. Furthermore, we adopt the
shorthand $$\eps_n = 1- a_n,$$ 
assume throughout that $\eps_n \to 0$, and define random measures $\xi^m_n$ and
$\txi^m_n$ by 
\begin{align}
\label{equ:def-ximn}
\xi^m_n := \eps_n^2 H^m_n \quad  \text{and} \quad \txi^m_n :=\eps_n^2
\tH^m_n.
\end{align}
\item A quantity is said to be a \emph{universal constant} if it depends
only on the primitives $\seq{a_n}$, $\seqmn{\pi^m_n}$ and $\seq{\mu_n}$,
and on the time horizon $T$. Such a constant is 
denoted by the generic symbol $C$, which may change from use to use.  The Landau notation $o$
and $O$ 
refers to asymptotics in $n$, uniformly in all variables except, perhaps,  
the primitives and $T$.
\item With the convolutional moment-generating
functional $M[\cdot]$ defined in \eqref{equ:mgf-fxi} we set
\begin{align}
\label{equ:def-hn}
h^m_n[f](t):=
\oo{\eps_n} \Big( \mgf_{\txi^m_n}[f](t) -1 \Big) =
\oo{\eps_n} \Big( \mgf_{\tH^m_n}[\eps_n^2 f](t) -1 \Big)
\in (-\infty,\infty]
\end{align}
for $f\in \ssi$ and $t \in [0,T]$. 
When the function $f$ is clear from the context, we often omit it
from the notation and write,
for example, $h^m_n$ for $h^m_n[f]$.
\item 
With a slight abuse of notation, we define 
rescaled versions of all the notation introduced above; Greek
superscripts  $\tau,\sg,\dots$ should be substituted by 
$\ct, \cs, \dots$:
\begin{align}
  \label{tau}
  h^{\tau}_n := h^{\ct}_n, \, \pi^{\tau}_n := \pi^{\ct}_n, \, \pi_n^{(\sg, \tau]}
  := \pi^{(\cs, \ct]}_n, \text{ etc.}
\end{align}
\item 
For $\nu_1,\nu_2\in\smlf\pos$ and $T\geq 0$, we set
\begin{align*}
\woT(\nu_1,\nu_2) = \int_0^{T} \abs{F_{\nu_1}(t) - F_{\nu_2}(t)}\, dt,
\end{align*}
where
$F_{\nu_i} = \nu_i([0,\cdot ])$, $i=1,2$,   are distribution functions of
$\nu_1$ and $\nu_2$.
When $\nu_1$ and $\nu_2$ are probability measures with supports in $\zT$,
$\woT(\nu_1,\nu_2)$ coincides with the $1$-Wasserstein distance between
$\nu_1$ and $\nu_2$ (see, e.g., \cite[Proposition 2.17, p.~66]{San15}).
\end{itemize}

\begin{lemma}
\label{lem:daily}
If $\nu,\nu'\in \smf$ and $h\in \bv$,
then  $h*\nu \in \bv$ with $D (h*\nu) = Dh *
  \nu$ and
  \begin{align}
    \label{equ:bd-2}
    \n{h*\nu}_{\bv} &\leq \n{h}_{\bv} \ns{\nu}
  \end{align}
  as well as
  \begin{align}
    \label{equ:bd-3}
    \n{ h*(\nu -\nu')}_{\slo} &\leq \n{h}_{\bv} \, \woT (\nu,\nu').
  \end{align}
\end{lemma}
\begin{proof}
  We have $(h*\nu)(t) = (Dh*\nu)[0,t]$, which implies that
  $h*\nu\in \bv$ with $D (h*\nu) = Dh*\nu$.  Therefore,
  \begin{align*}
    \n{h*\nu}_{\bv} \movealign{4}=
    \n{ D(h*\nu)} =
    \n{ Dh*\nu}
    \\ &=
    \n{ Dh^+*\nu - Dh^-*\nu}
    \leq (Dh^+*\nu) + (Dh^-*\nu)
    \\ &\leq Dh^+\, \nu + Dh^- \,  \nu
    = \n{Dh} \n{\nu} = \n{h}_{\bv} \n{\nu}.
  \end{align*}
  To prove \eqref{equ:bd-3}, we observe that
  \begin{align*}
    (Dh *\nu)[0,t] = (\nu*Dh)[0,t] = (F_{\nu}*Dh)(t) \efor
    t\geq 0,
  \end{align*}
  where $F_{\nu}(t) = \nu[0,t]$ denotes the distribution function of $\nu$.
  Hence,
  \begin{align*}
    \abs{h*\nu(t) - h*\nu'(t)} \leq
    \int  \abs{F_{\nu}(t-u) - F_{\nu'}(t-u)}\,
  \abs{D}h(du) \efor t\geq 0,
  \end{align*}
  and so,
  \begin{equation*}
    \begin{split}
      \int_0^T \abs{ h*\nu(t) - h*\nu'(t)}\, dt
      \movealign{4} \leq \int_0^T \int \abs{F_{\nu}(t-u) - F_{\nu'}(t-u)}
      \, \abs{D}h(du)\, dt
      \\ &\leq \int \int_0^T \abs{F_{\nu}(t) - F_{\nu'}(t)} \, dt\,
      \abs{D}h(du)
      = \n{h}_{\bv} \woT(\nu,\nu').  \qedhere
    \end{split}
  \end{equation*}
\end{proof}

\begin{lemma}
\label{lem:diver}
There exists a universal constant $C$ such that
for $m\in\N_0, n\in\N$ and 
$f\in\ssi$ with $f\leq \tot (1-\dl)$ we have
\begin{enumerate}
\item $h^m_n\in \ssi$ and
  \begin{align}
    \label{equ:Bragi}
    \inf_{u\in \zT} f(u) \leq h^m_n(t) \leq C \dl^{-1/2} 
    \sup_{u\in \zT} f(u) \efor
    t\in [0,T].
  \end{align}
\item If, additionally,  $f\in \bv$, then $h^m_n \in \bv$
  and
  \begin{align}
    \label{equ:bnd-BV}
    \n{h^m_n}_{\bv} &\leq C \n{f}_{\bv}.
  \end{align}
\end{enumerate}
\end{lemma}
\begin{proof}
We pick $\delta \in (0,1)$ and $f\in\ssi$ with $f\leq \tot (1-\dl)$
as in the statement and set
\begin{align}
\label{equ:Zmn}
Z^m_n = \exp{ \drtr \abs*{\txi^m_n}_{\smlf\pos}}
\end{align}
with $\txi^m_n$ as defined in \eqref{equ:def-ximn}.
Note that the distribution of $Z^m_n$ does not depend on $m$.

Since $f*\txi^m_n \leq \tot (1-\dl) \abs*{\txi^m_n}$,
\eqref{equ:tail-1} with $k=1$ implies that
\begin{equation}
\label{equ:craft}
\begin{split}
  h^m_n(t) &\leq  \eps_n^{-1} \ee{ Z^m_n (f*\txi^m_n)(t)}\leq
  \Bprn{ \textstyle\sup_{u\in [0,T]} f(u) }
  \eps_n \ee{Z^m_n \ns{\tH^m_n}}\\ & \leq C \delta^{-1/2}
  \textstyle \sup_{u\in [0,T]} f(u) \efor t\in \zT.
\end{split}
\end{equation}
To obtain a lower bound, we use Jensen's inequality and
the identity \eqref{equ:mm-by-rhom}:
\begin{align*}
h^m_n(t) &\geq   \eps_n^{-1}
\prn{\exp{ \ee{ (f*\txi^m_n)(t)}} -1}
\geq \eps_n^{-1} \ee{ (f*\txi^m_n)(t)}
\\ &= (f*\rho^m_n)(t) \geq \inf_{ u\in\zT } f(u).
\end{align*}
To establish \eqref{equ:bnd-BV}, we pick
$0\leq r\leq s \leq T$  and observe that
\begin{equation}
\label{equ:abs-hn}
\begin{split}
  \abs{ h^m_n(s) - h^m_n(r)}
  \hspace{-5em} & \hspace{5em}
  \leq
  \eps_n^{-1} \ee{ \abs{\exp{  f * {\txi^m_n}(s)} - \exp{  f*{\txi^m_n}(r)}} }\\
  & \leq  \eps_n^{-1} \ee{
    Z^m_n
  \abs{ f*{\txi^m_n}(s) - f*{\txi^m_n}(r)}}\\
  &\leq  \eps_n^{-1}\prn{F^m_n(r,s) + \hF^m_n(r,s)},
\end{split}
\end{equation}
where
\begin{align*}
F^m_n(r,s) &=
\ee{
  Z^m_n
\int_0^r \abs{f(s-u)-f(r-u)}\, \txi^m_n(du)}, \eand \\
\hF^m_n(r,s) &=    \ee{ Z^m_n \int_r^s \abs{f(s-u)}\, \txi^m_n(du)}.
\end{align*}
Since $\abs{f(b) - f(a)} \leq \int \ind{(a,b]}\, d\abs{D}f$, 
for all $a<b$ in $\zT$, we have
\begin{align*}
F^m_n(r,s) \leq \ee{ Z^m_n \iint  \inds{v\in (r-u,s-u], u\leq r}\,
  \abs{D}f(dv)\,
\txi^m_n(du)} .
\end{align*}
Hence, for $\kappa\in (0,T)$,
\begin{align*}
\oo{\kappa} \int_0^{T-\kappa} F^m_n(r,r+\kappa)\, dr
\movealign{5}
=
\\ &= \ee{ Z^m_n
  \iint \oo{\kappa} \int \inds{ r\in [v+u- \kappa, v+u) \cap [T-\kappa,u]}\, dr
\, \abs{D}f(dv)\txi^m_n(du)} \\
&\leq   \ee{  Z^m_n
\iint \abs{D}f(dv) \txi^m_n(du)  } \leq \n{f}_{\bv} \ee{ Z^m_n
\abs*{\txi^m_n}} \leq \eps_n C \n{f}_{\bv},
\end{align*}
where the last inequality follows from \eqref{equ:tail-1} with $k=1$. 
Similarly,
\begin{align*}
\oo{\kappa} \int_0^{T-\kappa} \hF^m_n(r,r+\kappa)\, dr &\leq
\oo{\kappa} \int_0^T \n{f}_{\ssiT} \ee{ Z^m_n \txi^m_n([u,u+\kappa])}\, du \\
&\leq
\n{f}_{\ssi}  \ee{Z^m_n  \int \oo{\kappa} \int \inds{s\in [u,u+\kappa]}\, du\,
\txi^m_n(ds) }
\\ & \leq  \eps_n C \n{f}_{\ssi} \leq \eps_n C \ns{f}_{\bv}.
\end{align*}
Therefore,
\begin{align}
\label{equ:equi-int}
\oo{\kappa} \int_0^{T-\kappa} \ns{h^m_n(u+\kappa) - h^m_n(u)}\, du &\leq
C  \ns{f}_{\bv}
\end{align}
uniformly in $m\in\N_0$, $n\in\N$ and $\kappa \in (0,T]$.
By \cite[Corollary 2.51, p.~53]{Leo17},
for each $m\in \N_0$ and $n\in\N$ there exists a 
signed measure $\iota^m_n$ on $(0,T]$
such that $\n{\iota^m_n}_{\sms} \leq C \n{f}_{\bv}$ and
\[ h^m_n(t)  - h^m_n(0) = \iota^m_n((0,t]) \text{ a.e., } \eforall t\in
\zT.\]
Thanks to  Lemma \ref{lem:daily}, 
we have $f*\xi^m_n\in \bv$ a.s.~
for any $f\in \bv$, and, in particular, 
$\lim_{s \downto t} (f*\xi^m_n)(s) = (f*\xi^m_n)(t)$ a.s.
If, additionally, $f \leq \drtr$ we have, as above,
\[ \sup_{t} \abs{\exp{(f*\txi^m_n)(t)} - 1} \leq C Z^m_n \ns{\txi^m_n}
\in \lone,\]
so that the dominated convergence theorem
can be used to conclude that the function $h^m_n$ is right continuous. 
This implies that $h^m_n(t) - h^m_n(0) =\iota^m_n((0,t])$ everywhere
and, consequently, that $h^m_n \in \bv$. 
Finally, by \eqref{equ:Bragi}, we have
\[ \n{h^m_n}_{\bv} \leq
\n{h^m_n - h^m_n(0)}_{\bv} + \abs{h^m_n(0)} \leq
C \n{f}_{\bv}. \qedhere\]
\end{proof}

\begin{lemma}
\label{lem:aswim0}
For all $m \in \N_0$, $n \in \N$, and $f \in \bvl $ with $\sup f <
1/2$,  we have
\begin{align}
  \label{equ:master}
h^m_n &= \sum_{k\geq m}\eps_n 
  (1-\eps_n)^{k-m} \bprn{f + (h^{k}_n)^2} *\pimk_n + r^m_n
  \ewhere
\n{r^n_m}_{\bv} = o(1).
\end{align}
\end{lemma}
\begin{proof}
Thanks to the relation \eqref{equ:hm-e} of Lemma 
\ref{lem:albee}, we have
\begin{align*}
\eps_n^2 f + \eps_n (1-\eps_n) h^{m+1}_n*\pi^{m+1}_n  =  \log(1+\eps_n h^m_n)
= \eps_n h^m_n - \tot \eps_n^2 \theta_n^m (h^m_n)^2 
\end{align*}
where $\theta_n^m = 1 + o(1)$. 
Therefore, 
\begin{align}
\label{equ:hmn-equ0}
h^m_n - (1-\eps_n) h^{m+1}_n*\pi^{m+1}_n =\eps_n  f +\tot \eps_n \theta^m_n
(h^m_n)^2.
\end{align}
Given $k\geq m$, we replace $m$ by $k$ in
\eqref{equ:hmn-equ0}, convolve it with
$\pimk_n$ and multiply by $(1-\eps_n)^{k-m}$ to obtain:
\begin{multline*}
(1-\eps_n)^{k-m} h^{k}_n*\pimk_n - (1-\eps_n)^{k+1-m} h^{k+1}_n*\pimkp_n = \\  =
\eps_n
(1-\eps_n)^{k-m} f*\pimk_n +
\tot
\eps_n (1-\eps_n)^{k-m} \theta^{k}_n (h^{k}_n)^2*\pimk_n.
\end{multline*}
Summing over $k\geq m$ gives
\begin{equation*}
\begin{split}
h^m_n &= \sum_{k\geq m} \eps_n
  (1-\eps_n)^{k-m} \bprn{f + \theta^{k}_n (h^{k}_n)^2} *\pimk_n.
\end{split}
\end{equation*}
It remains to show that 
\begin{align*}
r^m_n := \tot \sum_{k \geq m} \eps_n (1-\eps_n)^{k-m}  
(\theta^{k}_n-1) (h^{k}_n)^2*\pi^{(m,k]}
\end{align*}
converges to $0$ in $\bv$, uniformly in $m$. For that, we  
observe that  
\begin{align}
\label{apsis}
\sup_{m} \abs{r^m_n}_{\bv} &\leq
\tot
\sup_{m} \abs{\theta^k_n-1}
\sup_{m\geq 0, k \geq m} \n{(h_n^{k})^2*\pi^{(m,k]}_n}_{\bv}.
\end{align}
Since $\theta^m_n = 1 + o(1)$, 
it will suffice to show 
that the second factor in the product on the right in \eqref{apsis} is 
uniformly bounded in $n$.
This follows from the inequality
\eqref{equ:bd-2}, estimate  \eqref{equ:bnd-BV}
and the observation that $(h^m_n)^2\in \bv$ and
\begin{equation}
\label{equ:sq}
\n{(h^m_n)^2}_{\bv} \leq  2 \n{h^m_n}_{\ssi} \n{h^m_n}_{\bv} \leq
2 \n{h^m_n}_{\bv}^2 \leq C
\n{f}_{\bv}^2. \qedhere
\end{equation}
\end{proof}

\noindent Recall that a \define{modulus of continuity} is a 
function $\omega:
(0,\infty) \to \pos$ such that $\lim_{\delta \to 0+} \omega(\delta) =0$, and that
the notation $h^{\tau}$ was introduced in \eqref{tau} above. 
\begin{lemma}
\label{shack}
Suppose that Assumption \ref{A} holds and 
that $f \in \bv$ satisfies $\sup f < 1/2$. 
Then there exists a modulus of continuity $\omega$
such that
\begin{align*}
  \limsup_n \n{ h^{\tau}_n - h^{\sg}_n}_{\slo} \leq \omega(\tau-\sg) 
  \eforall \tau > \sg \geq 0. 
\end{align*}
\end{lemma}
\begin{proof}
  For $m \leq M \in \N$
  the representation \eqref{equ:master} of Lemma \ref{lem:aswim0} gives
  \begin{align*}
    h^M_n * \pi_n^{(m,M]}   
    &= (1-\eps_n)^{M-m} \Bprn{
      h^m_n + \sum_{k=m}^{M-1} \eps_n(1-\eps_n)^{k-m} (f+\tot (h^k_n)^2)*
  \pi_n^{(m,k]}} + R^{m,M}_n,
  \end{align*}
  where $\ns{R^{m,M}_n}_{\bv} = o(1)$. 
  The uniform boundedness of $\arr{h^m_n}$ in $\ssi$ yields the 
  following bound: 
\begin{align*}
    \n{ h^M_n * \pi_n^{(m,M]}  -
    (1-\eps_n)^{M-m} 
  h^m_n}_{\slo} 
  \hspace{-10em} & \hspace{10em} \leq C  \Big( \sum_{k=m}^{M-1}\eps_n(1-\eps_n)^{k-m}
  + \n{R^{m,M}_n}_{\slo} \Big)\\
 & = C( 1 - (1-\eps_n)^{M-m} ) + o(1).
  \end{align*}
  Substituting $m=\cs$ and $M = \ct$, we get 
\begin{align}
  \label{ropes}
  \n{h^\tau_n  - h^\sg_n}_{\slo} & \leq 
  p_n \n{ h^\sg_n}_{\slo} + 
  \ns{ h^\tau*\pi_n^{(\sg,\tau]} - h^\tau}_{\slo} + C p_n, 
\end{align}
where, as in \eqref{unode},  
\begin{align*}
p_n := 1 - (1-\eps_n)^{\ct - \cs}
= 1-e^{-(\tau - \sg)} + o(1) \leq (\tau - \sg) + o(1). 
\end{align*}
Boundedness of $\arr{h^m_n}$ in $\bv$ and the inequality
\eqref{equ:bd-3} of Lemma \ref{lem:daily} imply that
\begin{align*}
  \limsup_n \ns{ h^\tau*\pi_n^{(\sg,\tau]} - h^\tau}_{\slo} 
  &\leq C \limsup_n  \woT(\pi_n^{(\sg,\tau]}, \delta_0) = 
 C \limsup_n  \int_0^T t \pi_n^{(\sg, \tau]}(dt).
\end{align*}
Thanks to 
Assumption \ref{A} 
and the Portmanteau theorem,
the last $\limsup$ is bounded from above by 
  $\omega_0(\tau - \sg)$, where $\omega_0(\zeta) = 
   \int_0^T t \nu^{\zeta}(dt)$. 
The fact that $\lim_{\zeta \to 0} 
\nu^{\zeta} = \delta_0$ implies that $\omega_0$ is a modulus of continuity
and, then, so is 
$\omega(\zeta) = C( \omega_0(\zeta) + \zeta)$.
\end{proof}
\begin{proposition}
\label{paned}

Suppose that Assumption \ref{A} holds
and that $f \in \bv$ satisfies $\sup \abs{f} < 1/2$. 
Then the sequence $h^0_n$ converges in  $\slone$ to  
the unique solution $h \in \bv$ of the equation 
\eqref{equ:h-eq} with $\hrho = (1-\log(\hnu))^{-1}$. 
\end{proposition}
\begin{proof}
  For each $\tau \in \pos$ the sequence $\seq{ h^{\tau}_n}$ is bounded in
  $\bv$  by \eqref{equ:bnd-BV} of Lemma \ref{lem:diver}. Therefore, the
  compact containment of $\bv$-bounded sets in $\slo$ (see \cite[Theorem
  5.5, p.~203]{EvaGar92}) implies that $\seq{h^{\tau}_n}$ is relatively
  compact in $\slo$. Together with the asymptotic equicontinuity
  established in Lemma
  \ref{shack}, this allows us to apply the Arzel\`a-Ascoli theorem to the
  sequence $\seq{h_n}$, where each $h_n$ is  interpreted as a function
  from $\pos$ to $\slo$. Each subsequence of $\seq{h_n}$, therefore, admits
  a further subsequence along which $h^{\tau}_n$ converges in $\slo$,
  uniformly for $\tau$ in compact sets.  

We choose a convergent subsequence of $\seq{h_n}$, which we 
do not relabel, and 
denote its limit by $\tih$. Since
$\seq{h^{\tau}_n}$ is bounded in $\bv$ and converges in $\slo$ to
$\tih^{\tau}$, 
we have $\tih^{\tau} \in \bv$ (see 
\cite[Theorem 5.2, p.~199]{EvaGar92}) and 
\begin{align}
  \label{Trias}
  \ns{\tih^{\tau}}_{\bv} \leq \liminf_n\ns{h^{\tau}_n}_{\bv} \leq
  C\ns{f}_{\bv} \eforall \tau\geq 0.
\end{align}

The equation \eqref{equ:master} of Lemma \ref{lem:aswim0} with $m=\ct$ can
be rewritten as
  \begin{align}
    \label{prel}
    h^{\tau}_n = \int_{\eps_n\ct}^{\infty} (1-\eps_n)^{\tcs - \tct}
    \Big( f + \tot (h^{\sg}_n)^2\Big)* \pi_n^{( \tau, \sg]}\, d\sg +
      r^{\tau}_n
  \end{align}
  where $\sup_{\tau} \n{r^{\tau}_n}_{\slo} = o(1)$.
As in \eqref{equ:sq}, the family $\sets{ f + \tot (h^{\sg}_n)^2}{\sg \in
\pos, n \in\N }$ admits a uniform $\bv$-bound. Moreover,
the weak convergence $\pi_n^{(\tau, \sg]} \to
\nu^{\sg-\tau}$ of Assumption \ref{A} implies 
that $\woT(\pi_n^{(\tau,\sg]}, \nu^{\sg-\tau}) \to 0$, and so, 
\eqref{equ:bd-3} and \eqref{Trias} yield
\begin{align*}
    \Big(f + \tot (\tih^{\sg})^2\Big)* \pi_n^{( \tau, \sg]} \to 
    \Big(f + \tot (\tih^{\sg})^2 \Big)* \nu^{\sg -  \tau} \text{ in }\bv.
\end{align*}
Since $h^\sg_n \to \tih^\sg$ in $\slo$ and the sequence $\seq{h^{\sg}_n}$
is bounded in $\ssi$, we also have 
\begin{align*}
    \Big(f + \tot (h_n^{\sg})^2 \Big)* \nu^{\sg -  \tau} \to 
    \Big(f + \tot (\tih^{\sg})^2 \Big)* \nu^{\sg -  \tau} \text{ in } \slo.
\end{align*}
By \eqref{unode}, 
$(1-\eps_n)^{\cz}$ converges to  $e^{-\zeta}$ from below for each $\zeta$.
Consequently, 
the integrand in \eqref{prel} admits
a uniform $\slo$-bound of the form $C e^{-\sg}$. 
We may therefore
use the dominated convergence theorem to conclude that $\tih$ satisfies 
\begin{equation}
  \label{ekoeko}
  \begin{split}
  \tih^{\tau} & = \int_\tau^\infty e^{-(\sg-\tau)} 
  \Big( f + \tot (h^{\sg})^2\Big) * \nu^{\sg -  \tau}\, d\sg\\
  & = \int_0^\infty 
  \Big( f + \tot (h^{\zeta+\tau})^2\Big) * \nu^{\zeta} e^{-\zeta} \, d\zeta
   \eforall \tau \geq 0.
\end{split}
\end{equation}
  By Lemma \ref{lem:sys-sol}, 
 or by the general theory of Volterra equations (see, e.g., \cite{Mil71}),
 the equation \eqref{equ:h-eq} admits a unique solution  in $\ssi$ which we
 denote by $h$. 
With $\hh$ denoting the Laplace transform  of $h$ and 
$\psi = -\log(\hnu)$, we have
\begin{align*}
  \hh = \frac{(f + \tot h^2)^{\lap}}{1+\psi} = 
  \int_0^{\infty} e^{-\sg(1+\psi)} (f+\tot h^2)^{\lap} \, d\sg 
  =
  \prn{\int_0^{\infty} e^{-\sg(1+\psi)} (f+\tot h^2) \, d\sg
  }^{\lap},
\end{align*}
and it follows that the constant function $h^{\tau} = h$ solves
\eqref{ekoeko}.

Our next task is to show that $\tih^{\tau} = h$ for all $\tau$. 
To do so, we prove
that \eqref{ekoeko} admits a unique solution in the class of all measurable
functions $\tau \mapsto g^{\tau}$ with $\sup_{\tau} \n{g^{\tau}}_{\ssi}
< 1$. 
Thanks to the boundedness assumption we made
on the function $f$ in the statement, both the constant function $h$ and
the limit $\tih$ have this property. Indeed, in the case of $h$ it follows
from Lemma \ref{lem:palma}, while in the case of $\tih$ it is inherited
from the sequence $\seq{h_n^{\tau}}$  via the bound \eqref{equ:tail-0} of
Proposition \ref{pro:dhava}. 

Suppose that 
there exists two solutions, $g_1$ and $g_2$, of \eqref{ekoeko} such 
that
$\ns{g^{\tau}_1}_{\ssi},
\ns{g^{\tau}_2}_{\ssi} \leq 1-\kappa$ for all $\tau>0$ and some
$\kappa>0$.
The function
$M(\tau) := \n{g^{\tau}_1 - g^{\tau}_2}_{\ssi}$ satisfies the inequality
\begin{align*}
M(\tau) \leq
\int_{\tau}^{\infty} 
  \tot\Big(\ns{g^{\sg}}_{\ssi}+\ns{\tg^{\sg}}_{\ssi}\Big)
  M(\sg)
  e^{-(\sg-\tau)}
  \, d\sg
  \leq  (1-\kappa) \int_{\tau}^{\infty} M(\sg)
  e^{-(\sg-\tau)}\, d\sg.
\end{align*}
The last integral averages the function $M$
over $[\tau,\infty)$ so
for each $\tau$ we can find $\tau' \geq \tau$ such that 
$M(\tau) \leq (1-\kappa/2) M(\tau')$.
Iterating this procedure
and starting from any $\tau_0 =\tau $, 
we can find a nondecreasing sequence $\seq{\tau_n}$ such that
$M(\tau) \leq (1-\kappa)^n M(\tau_n) \leq 2 (1-\kappa)^n$. This
implies that $M(\tau) = 0$ for each $\tau$, i.e., 
that $g_1 = g_2$.
Hence, each convergent subsequence of $\seq{h_n}$ converges to the same
limit, so, by relative compactness, we may conclude that the entire sequence
converges to the constant function $ \tau \mapsto h$. 
\end{proof}
\begin{proof}[Conclusion of the proof of Theorem \ref{thm:main}]
Let $p>1$ and $\dl \in (0,1)$ be such that
$ p \sup \abs{f} \leq \tot (1-\delta)$.
The expression \eqref{equ:sL-H-h} for the moment-generating function
of $H_n$ implies that, for $T \geq 0$, 
\begin{equation}
  \label{Bassa}
  \begin{split}
  \log \ee{ \prn{ \Exp{ (f*\xi_n)(T)}}^p}
  \movealign{8} \leq  \log M_{H_n}[\drtr \eps_n^2  ](T) =
\prn{(M_{\tH_n}[ \eps_n^2 \drtr  ] - 1)*\mu_n}(T)
\\ & \leq
\prn{\ee{ \Exp{ \eps_n^2 \drtr \tH_n[0,\cdot] }-1}*\mu_n}
(T) \leq \eps_n^{-1} \ee{ \Exp{ \eps_n^2 \drtr \abs{\tH_n} }-1}\,  
\eps_n \mu_n[0,T].
  \end{split}
\end{equation}
Estimate \eqref{equ:tail-0} of Proposition \ref{pro:dhava} and the
fact that $\eps_n \mu_n$ converges weakly on $[0,T]$ imply that the 
last expression in \eqref{Bassa} is bounded in $n$. Hence,  
\begin{align}
  \label{UI}
  \Bsets{\bexp{(f*\xi_n)(T)}}{n\in\N} \text{ is a uniformly 
  integrable family.}
\end{align}

In particular, the sequence $\seq{ \ee{ \xi_n[0,T]}}$ is bounded.
Therefore, the standard tightness criterion
(see, e.g.,  \cite[Theorem 4.10, p.~118]{Kal17})
for weak convergence of random measures,
implies that the sequence
$\seq{\xi_n} := \seq{\xi^0_n}$ is tight on $[0,T]$. 

We pick an arbitrary convergent subsequence of $\seq{\xi_n}$, and,
taking the usual liberty of not relabeling the indices, we
denote this sequence by $\seq{\xi_n}$ as well, and
its limit by $\xi$. 
For $f \in \bv$ with $\sup \abs{f} < 1/2$, 
let $\{ h^m_n\}_{m,n} = \{ h^m_n[f] \}_{m,n}$ be the array associated
with $f$ as in \eqref{equ:def-hn} above.
Proposition \ref{paned} states that $h^0_n \to h$ in $\slo$, 
where $h \in \bv$ is the
unique solution of \eqref{equ:h-eq}.
The weak convergence $\eps_n\mu_n \to \mu$ on $[0,T]$ 
implies that $\woT(\eps_n\mu_n,\mu) \to 0$ for each $T>0$, and we 
can use \eqref{equ:bd-3} to obtain 
\begin{align*}
\n{\eps_n h_n *  \mu_n - h*\mu}_{\slo} &\leq
\n{(h_n - h)*\eps_n \mu_n}_{\slo}
+ \n{h*(\eps_n \mu_n - \mu)}_{\slo}
\\ & \leq C \n{h_n - h}_{\slo} + \n{h}_{\bv} \woT(\eps_n \mu_n, \mu) \to 0.
\end{align*}
Therefore, if, additionally, $f\in C_0 \cap \bv$,
the convergence $\xi_n \Rightarrow \xi$ and \eqref{UI}
imply that
\begin{align}
\label{equ:xif}
\mgf_{\xi}[f] = \lim_n  \mgf_{\xi_n}[f] = \lim_n \Exp{  \eps_n h_n *
\mu_n } = \Exp{ h*\mu } \text{ a.e.},
\end{align}
where, if necessary, we pass to a subsequence to  
guarantee a.e.-convergence of $\eps_n h_n*\mu_n$ to $h*\mu$.
The estimates in \eqref{equ:tail-1} and 
\eqref{Bassa} above
imply that asdf 
\begin{align*}
  \liminf_n \log\ee{ e^{\drtr \xi_n[0,T]}} \leq 
  \liminf_n C \eps_n^{-1} 
  \ee{ e^{ \tot \eps_n^2 (1-\delta) \abs{\tH_n}}-1} \leq C(1-\sqrt{\dl})
\end{align*}
for some $C$ independent of $n$ and $\dl$. 
The Portmanteau theorem then guarantees that $\ee{ \exp{ \drtr
\xi[0,T]}} \leq C$ for all $\dl>0$, and, consequently,  that
\begin{align}
  \label{mmmm}
  \ee{ \exp{ \tot \xi[0,T] }} \leq C < \infty. 
\end{align}
The convolution $f*\xi$ is in $\bv$ and, therefore, right continuous.
The random variable $\tot \xi[0,T]$ is an upper bound for $f*\xi$, 
so, thanks to \eqref{mmmm} 
we can use the dominated convergence theorem to conclude that 
this right continuity is inherited by
$M_{\xi}[f]$.
Consequently,
\eqref{equ:xif} can be strengthened to a pointwise equality
\begin{align}
\label{equ:Kathy}
\mgf_{\xi}[f] = \exp{ h[f]*\mu} \eforall f\in C_0 \cap \bv \ewith
\sup \abs{f} < 1/2. 
\end{align}
Thanks to \cite[Corollary 2.3, p.~53]{Kal17} enhanced by an 
additional, standard
approximation step, the class of functions $f$ in \eqref{equ:Kathy} is
rich enough 
to fully determine the law of $\xi$ on $[0,T]$. Therefore, all 
convergent subsequences of the original sequence $\seq{\xi_n}$ converge 
in distribution to the same limit. By tightness, this in turn
implies that the original sequence $\seq{\xi_n}$ 
converges weakly on $[0,T]$ to the random measure $\xi$ characterized 
by \eqref{equ:Kathy}.

\smallskip

The next part of the proof extends the characterization \eqref{equ:Kathy}
of the limit to the class of all functions $f \in C_0$ with $\sup \abs{f}
\leq 1/2$. Given such an $f$, we start by 
choosing a sequence $\seqk{f_k}$ in $C^1$ with
$f_k(0)=0$ such that $\sup \abs{f_k}< 1/2$ 
and $\n{f-f_k}_{\ssi} \to 0$.
Since $C^1 \subseteq C_0 \cap \bv$, each $f_k$  belongs to the class 
defined in \eqref{equ:Kathy}. On the one hand, the dominated convergence
theorem implies, via \eqref{mmmm}, that
$M_\xi[f_k](t) \to M_{\xi}[f](t)$ for all $t \in [0,T]$. 
On the other hand, Proposition \ref{pro:tangi} states that
\[ \int_0^T \n{h[f_k] - h[f]}_{\ssit}\, dt \to 0. \]
The monotonicity of
$t \mapsto \n{h[f_k] - h[f]}_{\ssit}$ implies that $h[f_k] \to h[f]$
uniformly on each compact subset of 
$[0,T)$.  Since $\mu$ is locally bounded and
$\seqk{h[f_k]}$ admits a uniform $\ssi$ bound, Young's inequality 
implies that 
$h[f_k]*\mu \to h[f]*\mu$
pointwise on $[0,T)$, and, consequently, that $M_{\xi}[f] = \exp{ h[f]*\mu}$
on $[0,T)$. 

\smallskip

The remaining step, namely  passing from a finite horizon $[0,T]$ to
$\pos$, is achieved by the standard diagonalization procedure over a sequence
of intervals $[0,T_k]$ with $T_k \to \infty$.
\end{proof}

\bibliographystyle{amsalpha}
\bibliography{Hawkes}       
\end{document}